\documentclass[12pt,xcolor=svgnames]{amsart}

\usepackage{amsmath, amsthm, amssymb, amsfonts, verbatim}

\usepackage{xspace,mathtools,bm}

\usepackage[svgnames]{xcolor}

\usepackage[pagebackref=true, colorlinks=true, citecolor=blue]{hyperref}

\allowdisplaybreaks

\hypersetup{
    colorlinks=true,
    citecolor=red,
    linkcolor=blue,
    filecolor=magenta,      
    urlcolor=cyan,
    pdfpagemode=FullScreen,
    }
\newtheorem{theorem}{Theorem}[section]
\newtheorem{corollary}[theorem]{Corollary}
\newtheorem{lemma}[theorem]{Lemma}
\newtheorem{definition}[theorem]{Definition}
\newtheorem{proposition}[theorem]{Proposition}

\newcommand{\R}{\mathbb{R}}
\newcommand{\eps}{\varepsilon}
\newcommand{\abs}[1]{\mid\!#1\!\mid}
\newcommand{\norm}[2]{\left\lVert#1\right\rVert_{#2}}
\newcommand{\eqdef}{:=}

\newcommand{\Obsolete}[1]{
    }

\DeclareMathOperator\supp{supp}

\DeclareFontFamily{U}{mathx}{}
\DeclareFontShape{U}{mathx}{m}{n}{<-> mathx10}{}
\DeclareSymbolFont{mathx}{U}{mathx}{m}{n}
\DeclareMathAccent{\widehat}{0}{mathx}{"70}
\DeclareMathAccent{\widecheck}{0}{mathx}{"71}

\date{\today}
\usepackage[english]{babel}

\author{Nicholas Harrison}
\address{Oregon State University, Department of Mathematics}
\email{harrnich@oregonstate.edu}

\author{Zachary Radke}
\address{Oregon State University, Department of Mathematics}
\email{radkeza@oregonstate.edu}
\title[Euler Equations in Function Spaces of Generalized Smoothness]{Well-Posedness for the Euler Equations in Function Spaces of Generalized Smoothness}

\begin{document}

\maketitle

\begin{abstract}
{We consider the question of well-posedness for the incompressible Euler equations in generalized function spaces of the type $B^{s,\psi}_{p,q}(\R^d)$ and $F^{s,\psi}_{p,q}(\R^d)$ where $\psi$ is a slowly varying function in the Karamata sense and $s=d/p+1$. We prove that if $\psi$ grows fast enough, then there is a local in time solution to the Euler equations. We also establish a BKM-type criterion that allows us to obtain global existence in two dimensions.}
\end{abstract}

\section{Introduction}\label{Section Intro}
\subsection{Background}
We consider the persistence of regularity of solutions to the incompressible Euler equations, given by
\begin{equation}
\begin{cases}\label{euler velocity}
\tag{$E$}
\partial_tu + (u \cdot \nabla)u=-\nabla p &\text{ in } (0,T) \times \R^d,\\
\text{div }u=0 &\text{ in } [0,T) \times \R^d,\\
u \vert_{t=0}=u_0 &\text{ in } \R^d,
\end{cases}
\end{equation}
where $u:[0,T)\times\R^d\to \R^d$ represents the velocity field of an ideal fluid and $p:[0,T) \times \R^d \to \R$ denotes the scalar pressure. Using the incompressibility condition of the velocity field, we can determine the pressure from the velocity as $p=(-\Delta)^{-1}\text{div} (u \cdot \nabla u)$. We also define the vorticity of the fluid as a matrix with $ij$-entry $\omega_{ij}=\partial_ju_i-\partial_iu_j$.  The velocity can be recovered from the vorticity using the relation $u_j=\partial_i(-\Delta)^{-1}\omega_{ij}$ where we sum over repeated indices. We often abuse notation and write this relation as $u=\text{div}((-\Delta)^{-1}\omega)$. This relation is known as the Biot-Savart law. We then have the vorticity equation
\begin{equation}
\begin{cases}\label{euler vorticity}
\tag{$V$}
\partial_t\omega+u\cdot\nabla\omega=\omega \cdot \nabla u &\text{ in } (0,T) \times \R^d,\\
u=\text{div}((-\Delta)^{-1}\omega) &\text{ in } [0,T) \times \R^d,\\
\omega \vert_{t=0} = \omega_0 &\text{ in } \R^d.
\end{cases}
\end{equation} 
In two dimensions, we see the vorticity depends only on $\omega_{21}=\partial_1u_2-\partial_2u_1$. This coincides with the notion of treating the two-dimensional velocity field as three-dimensional, with the $z$-component being zero. Then, taking the curl shows that the vorticity only has a $z$-component given by $\omega_{21}$. In turn we see the vorticity must be orthogonal to the velocity gradient in two dimensions. This greatly simplifies the structure of \eqref{euler vorticity} to the transport equation
\begin{equation}
\begin{cases}\label{euler 2d vorticity}
\tag{$V_2$}
\partial_t\omega+u\cdot\nabla\omega=0 &\text{ in } (0,T) \times \R^2,\\
u=\nabla^{\perp}(-\Delta)^{-1}\omega &\text{ in } [0,T) \times \R^2,\\
\omega \vert_{t=0} = \omega_0 &\text{ in } \R^2.
\end{cases}
\end{equation} 
The solutions to \eqref{euler 2d vorticity} are given by $\omega(t,x)=\omega_0(t,X_t^{-1}(x))$ where $X_t^{-1}$ is the inverse of the flow map $X_t(x)$, which satisfies the ODE
\begin{equation}
\label{Flow Map ODE}
\begin{cases}
\frac{d}{dt}X_t(x)=u(t,X_t(x))\quad &\text{in } (0,T)\times\mathbb{R}^2,\\
X_t(x) \vert_{t=0} = x\quad &\text{in } \mathbb{R}^2.
\end{cases}
\end{equation}
When studying the well-posedness of \eqref{euler velocity}, a common strategy is to study weak solutions and then show that the weak solution is sufficiently smooth. 
\begin{definition}
\label{weak solution}A vector field $u: [0,T) \times \R^d \to \R^d$ is a weak solution to \eqref{euler velocity} with initial data $u_0:\R^d  \to \R^d$ if 
\begin{enumerate}
\item $u \in L^2_{loc}([0,T)\times \R^d)$,\\
\item 
$$
\int_0^T\int_{\R^d} \left\{ u\partial_t\phi + u_ku_l\partial_{x_l}\phi_k\right\}\,dx\,dt = \int_{\R^d}\left\{ u(T) \cdot \phi(T)-u(0)\cdot\phi(0)\right\}\,dx
$$
for all $\phi \in C_0^{\infty}([0,T);\mathbb{R}^d)$ with $\text{div }\phi=0$, where we sum over repeated indices, and
\item $\text{div }u=0$ in $\mathcal{D}'(\R^d)$.
\end{enumerate}
\end{definition}
\subsection{Prior Work}
The study of \eqref{euler velocity} in various function spaces has been extensively investigated. One of the first results was from \cite{L25}, where the author establishes local well-posedness in the spaces $C^{k,\gamma} \cap L^p(\R^2)$ where $k \geq 1, 0< \gamma <1$ and $p \in (1,\infty)$. Global well-posedness in these spaces was then established in \cite{WW33}. In the context of Sobolev spaces, local well-posedness in $W^{s,p}(\R^2)$ was established in \cite{KP86} with $s>\frac{2}{p}+1$ and $p \in (1,\infty)$. These solutions are then extended to global solutions using the BKM criterion \cite{BKM84}, which states that a solution on $[0,T]\times \mathbb{R}^d$ in $W^{s,p}(\R^d)$ can be extended if and only if $\norm{\omega}{L^1([0,T];L^{\infty})}<\infty$. We also note there are short-time analogues in $\R^3$ to the above results. 
\\
\\
The question about regularity in critical Sobolev spaces $W^{s,p}(\R^2)$ with $s=\frac{2}{p}+1$ and $p \in (1,\infty)$ remained unanswered for quite some time after. Here, criticality can be understood in the sense of being barely unable to close the estimate $\norm{\nabla u}{L^{\infty}} \leq C\norm{u}{W^{s,p}}$. In \cite{BL15}, the authors establish strong ill-posedness in these function spaces. The dichotomy between well-posedness and strong ill-posedness at the critical regime has inspired work in Besov and Triebel-Lizorkin spaces, as these spaces offer finer embeddings. In \cite{DC03} and \cite{DC04} the author establishes local well-posedness in the spaces $B^s_{p,q}(\R^d)$ and $F^{s}_{p,q}(\R^d)$ provided that $s>d/p+1$. Note that under these assumptions, the velocity field is
Lipschitz. In the critical case, local well-posedness is established in $B^{d/p+1}_{p,1}(\R^d)$ and $F^{d+1}_{1,q}(\R^d)$ and global well-posedness is shown when $d=2$ (see \cite{MV98} and \cite{DC03}). There have also been ill-posedness results in scaling critical Besov and Triebel-Lizorkin spaces, which were established in \cite{BL15}.
\\
\\
Another way to refine the scale of $H^s$ Sobolev spaces is to include an extra factor which imposes additional sub-power decay on the Fourier transform\footnote{Throughout this paper we use the convention 
$
    \hat{f}(\xi) := \int_{\mathbb{R}^d} f(x) e^{-2\pi i x \cdot \xi}\, dx $ to denote the Fourier transform of a function $f:\mathbb{R}^d\to\mathbb{R}$.}. In particular, we assume that $\psi$ is a slowly varying function in the Karamata sense (see definition \ref{slowly varying}) and define the space $H^{s,\psi}(\R^d)$ as the set of $f \in \mathcal{S}'(\R^d)$ such that the following norm is finite,
\begin{equation}
\label{Hormander spaces}
\norm{f}{H^{s,\psi}} = \left(\int_{\R^d} (1+|\xi|^2)^s\psi^2(|\xi|)|\hat{f}(\xi)|^2d\xi\right)^{\frac{1}{2}}.
\end{equation}
These spaces are known as Hörmander spaces. Using Proposition \ref{Karamata Representation} one can show that $\psi$ grows slower than any power function. It then follows that 
$$
\bigcup_{\eps>0} H^{s+\eps}(\R^d) \eqdef H^{s^+}(\R^d) \subset  H^{s,\psi}(\R^d) \subset H^s(\R^d),
$$
with all of the above containments being strict. When $\psi(t)=\log^{\alpha}(e+t)$, sharp estimates in related function spaces are established in \cite{BN19} and then applied to the study of continuity equations in \cite{BN21}. Later in \cite{CCS24} weak solutions to the two-dimensional Euler equations in $H^{0,\psi}(\mathbb{T}^2)$ are studied, and in \cite{MS22}, the existence of weak solutions to transport equations in $H^{0,\psi}(\mathbb{T}^2)$ is established provided that $u \in L^1((0,T):W^{1,p}(\mathbb{T}^2))$ and $\frac{1}{2}\leq \alpha \leq \frac{p}{2}$ for some $1<p<\infty$.
\\
\\
When studying the Besov and Triebel-Lizorkin variants of Hörmander spaces, we recall that the Littlewood-Paley operator roughly projects frequencies on the scale of $2^j$, so we replace the factor of $2^{js}$ by $2^{js}\psi(2^j)$. Again, in the case of the logarithm, we see that $\log^{\alpha}(e+2^j) \sim (1+j)^{\alpha}$. These spaces in which we replace $2^{js}$ by $2^{js}(1+j)^{\alpha}$ are notated as $B^{s,\alpha}_{p,q}(\R^d)$ and $F^{s,\alpha}_{p,q}(\R^d)$ and are studied extensively in \cite{DT18}. 
\\
\\
In this work, we consider the scaling critical ($s=d/p+1)$ Besov and Triebel-Lizorkin spaces indexed by a slowly varying function notated $B^{s,\psi}_{p,q}(\R^d)$ and $F^{s,\psi}_{p,q}(\R^d)$. We show that if $\psi$ grows sufficiently fast, then the Euler equations are locally well-posed in $\R^d$ for $d \geq 2$. We also prove a BKM-type inequality which allows us to continue our solution provided that $\norm{\omega}{L^1([0,T];\dot{B}^0_{\infty,1})}<\infty$. As an immediate corollary, we obtain global well-posedness when $d=2$. We remark that this result is sharp. In a forthcoming paper, we prove that \eqref{euler vorticity} is strongly ill-posed in $H^{1,\psi}(\R^2)$ where $\psi(t)=\log^{\alpha}(1+t)$ and $\alpha\in(0,\frac{1}{2}]$ which lies strictly above the ill-posedness result of \cite{BL15} and just below the spaces we consider in the well-posedness results below.
\subsection{Main Results}
We state our main results below. See Definitions \ref{Generalized function spaces} and \ref{good sv functions}  for the function spaces and admissible $\psi$ considered below.
\begin{theorem}[Local existence in $\R^d$ and blow up criterion]Let $d \geq 2$ and $X^{s,\psi}_{p,q}(\R^d)$ denote either $B^{s,\psi}_{p,q}(\R^d)$ or $F^{s,\psi}_{p,q}(\R^d)$. Suppose $s=d/p+1$, $p,q \in (1,\infty)$, and $\psi \in \mathcal{M}_{q'}$ if $X^{s,\psi}_{p,q}(\R^d) =B^{s,\psi}_{p,q}(\R^d)$ or $\psi \in \mathcal{M}_{p'}$ if $X^{s,\psi}_{p,q}(\R^d) =F^{s,\psi}_{p,q}(\R^d)$. Then, given $u_0 \in X^{s,\psi}_{p,q}(\R^d)$, there exists a $T>0$ and a unique solution $u$ to \eqref{euler velocity} such that 
$$ 
u \in C([0,T];X^{s,\psi}_{p,q}(\R^d)).
$$
Furthermore, the solution $u$ blows up at time $T^*>T$, namely 
$$
\limsup_{t \to T^*} \norm{u(t)}{X^{s,\psi}_{p,q}} =\infty,
$$
if and only if 
$$
\int_0^{T^*} \norm{\omega(t)}{\dot{B}^0_{\infty,1}}\,dt=\infty.
$$
\end{theorem}
As a corollary, we obtain global existence in two dimensions.

\begin{corollary}[Global existence in $\R^2$] Let $X^{s,\psi}_{p,q}(\R^2)$ denote either $B^{s,\psi}_{p,q}(\R^2)$ or $F^{s,\psi}_{p,q}(\R^2)$. Suppose $s=\frac{2}{p}+1$, $p,q \in (1,\infty)$, and $\psi \in \mathcal{M}_{q'}$ if $X^{s,\psi}_{p,q}(\R^2) =B^{s,\psi}_{p,q}(\R^2)$ or $\psi \in \mathcal{M}_{p'}$ if $X^{s,\psi}_{p,q}(\R^2) =F^{s,\psi}_{p,q}(\R^2)$. Then, given $u_0 \in X^{s,\psi}_{p,q}(\R^2)$, there exists a unique solution $u$ to \eqref{euler velocity} such that 
$$ 
u \in C([0,\infty);X^{s,\psi}_{p,q}(\R^2)).
$$
\end{corollary}
\subsection{Notation and Organization of Paper}
We next define some of the notation that we will use throughout this paper. Consider two quantities $A$ and $B$ parameterized by some index set $\Lambda$, we then say that $A \lesssim B$ if there exists some $C>0$ such that $A(\lambda) \leq CB(\lambda)$ for all $\lambda \in \Lambda$. If we wish to highlight the dependence of $C$ on a parameter, say $b$, we write $A \lesssim_b B$.  We write $A \sim B$ if $A \lesssim B$ and $B \lesssim A$. Furthermore, we let $C>0$ denote a positive constant that may change from line to line. If we wish to highlight the dependence of the constant $C$ on a parameter $b$, we write $C=C(b)$.
\\
\\
We organize the paper as follows. In the Section \ref{Section Prelim} we cover preliminaries such as Littlewood-Paley theory, function spaces, and slowly varying functions. In Section \ref{Section Besov} we establish an a priori estimate of smooth solutions to \eqref{euler velocity} in the spaces $B^{s,\psi}_{p,q}(\R^d)$. Using this a priori estimate, we construct a sequence of approximating solutions and then show that the limit of this sequence is a classical solution to \eqref{euler velocity}. Following, we prove a blow up criterion which allows us to obtain global well-posedness in two dimensions. In Section \ref{Section Triebel} we repeat the same argument but in the spaces $F^{s,\psi}_{p,q}(\R^d)$, which introduces separate difficulties. In the appendix, we prove many calculus inequalities. While these are known in the classical Besov and Triebel-Lizorkin spaces, we are not able to conclude without proof that these same estimates hold in these function spaces of generalized smoothness. The proofs follow quite similarly to those in the literature, but we include them for completeness.
\section{Preliminaries}\label{Section Prelim}

\subsection{Littlewood-Paley Theory} 
Before defining our function spaces, we begin with a brief overview of the Littlewood-Paley operators. We closely follow \cite{CH95}.\\
\\
It is well known that if $\mathcal{C}$ is the annulus of center $0$, inner radius $3/5$, and outer radius $5/3$, then there exists two positive radial functions $\hat{\chi} \in C_0^{\infty}(B_{5/6}(0))$ and $\hat{\varphi}\in C_0^{\infty}(\mathcal{C})$ such that, if $j \in \mathbb{Z}$, then we set $\varphi_j(x)=2^{jd}\varphi(2^jx)$ and
$$
\hat{\chi}+\sum_{j\geq 0} \hat{\varphi}_j = \hat{\chi} + \sum_{j \geq 0} \hat{\varphi}(2^{-j}\cdot)
 \equiv 1.$$
For $n\geq 0$ we define $\chi_n$ via its Fourier transform, $\hat{\chi}_n$, as
$$
\hat{\chi}_n(\xi) = \hat{\chi}(\xi) + \sum_{j \leq n} \hat{\varphi}_j(\xi) \text{ for } \xi \in \R^d.
$$
For a tempered distribution $f \in \mathcal{S}'(\R^d)$, we define the operator $S_n$ as 
$$
S_n f =\chi_n \ast f.
$$
Formally, $S_n$ can be thought of as a projection in Fourier space into the ball of radius of order $2^n$. For $f \in \mathcal{S}'(\R^d)$ we define the inhomogeneous Littlewood-Paley operators $\Delta_j$ by 
\begin{equation*}
\label{LP operator}
\Delta_j f
=
\begin{cases}
    0, \hspace{1cm} &\text{if }j < -1,\\
    \chi \ast f, &\text{if } j=-1, \\
    \varphi_j \ast f, &\text{if } j\geq 0,
\end{cases}
\end{equation*}
and for all $j\in \mathbb{Z}$ we define the homogeneous Littlewood-Paley operators $\dot{\Delta}_j$ by
$$
\dot{\Delta}_j f = \varphi_j \ast f.
$$
We also define the fattened Littlewood-Paley operator $\tilde{\Delta}_j$ as $\tilde{\Delta}_j=\Delta_{j-1}+\Delta_j+\Delta_{j+1}$. One easily verifies that $\Delta_j=\tilde\Delta_{j}\Delta_j$. Up to constants, one can formally think of $\dot{\Delta}_j$ as a projection in Fourier space into the annulus of inner and outer radius of order $2^j$. For any $u \in \mathcal{S}'(\R^d)$, we have that $u = \lim_{n \to \infty} S_nu$ where the limit holds in $\mathcal{S}'(\R^d)$. Or equivalently
$$
u = \sum_{j\geq -1} \Delta_j u \text{ in } \mathcal{S}'(\R^d).
$$
One benefit of using Littlewood-Paley operators can be seen in Bernstein's Lemma. Heuristically, Bernstein's Lemma states that if the support of $\hat{u}$ is included in some ball $B_{\lambda}(0)$, then a derivative of $u$ costs no more than a power of $\lambda$ when taking an $L^p$ norm. This will be made precise below. We let $\mathcal{C}_{a,b}(0)$ denote the annulus with inner radius $a$ and outer radius $b$.

\begin{lemma}(Bernstein's Lemma) Let $(r_1,r_2)$ be a pair of strictly positive real numbers such that $r_1<r_2$ and let $\lambda$ denote a positive real number. Then there exists a constant $C>0$, such that, for every integer $k$, every pair of real numbers $(a,b)$ satisfying $b \geq a \geq 1$, and every function $u \in L^a(\R^d)$,
\begin{enumerate}

\item 
\label{Bernstein in Ball}
If $\supp{\hat{u}} \subset B_{r_1\lambda}(0)$ then
\begin{equation*}
 \sup_{\abs{\alpha} = k} \norm{\partial^{\alpha}u}{L^b} \leq C^k\lambda^{k+d(\frac{1}{a}-\frac{1}{b})}\norm{u}{L^a}.
\end{equation*}\\
\item 
\label{Bernstein in Annulus}
Furthermore, if $\supp{\hat{u}} \subset \mathcal{C}_{r_1\lambda, r_2\lambda}(0)$ then
\begin{equation*}
 C^{-k}\lambda^k\norm{u}{L^a} \leq  \sup_{\abs{\alpha} = k} \norm{\partial^{\alpha}u}{L^a} \leq C^k\lambda^k\norm{u}{L^a}.
\end{equation*}
\end{enumerate}
\end{lemma}

Later on, it will be useful to decompose the product of two functions into their individual phase blocks. Formally, for two tempered distributions $f$ and $g$ we write their product as
$$
fg=\sum_{j,j'}\Delta_jf\Delta_{j'}g.
$$
We then split this sum into three different parts. The first is when the low frequencies of $f$ multiply the high frequencies of $g$, the second is when the low frequencies of $g$ multiply the high frequencies of $f$, and the third is when $j$ and $j'$ are comparable. This gives the nonhomogeneous Bony paraproduct decomposition.

\begin{definition}
    The nonhomogeneous paraproduct of $f$ by $g$ is defined by 
    $$
    T_gf = \sum_{j=1}^{\infty} S_{j-2}g\Delta_jf .
    $$
    The nonhomogeneous remainder of $f$ and $g$ is defined by 
    $$
    R(f,g) = \sum_{\substack{j,j'\\ |j-j'|\leq 1}}\Delta_jf\Delta_{j'}g.
    $$
    We then have the following nonhomogeneous Bony paraproduct decomposition
    \begin{equation}
    \label{paraproduct decomposition}
    fg=T_fg + T_gf + R(f,g) \text{ in } \mathcal{S}'(\R^d).
    \end{equation}
    We also define the homogeneous Bony paraproduct decomposition 
    \begin{equation}
    \label{homogeneous paraproduct decomposition}
    fg=\dot{T}_fg + \dot{T}_gf+\dot{R}(f,g) \text{ in } \mathcal{S}'(\R^d) /\mathcal{P}'(\R^d)
    \end{equation}
    by replacing the Littlewood-Paley operators in \eqref{paraproduct decomposition} with their homogeneous counterpart and where for $n \in \mathbb{Z}$ we define $\dot{S}_nf = \sum_{j=-\infty}^n \dot{\Delta}_jf$.
    
\end{definition}

\subsection{Function Spaces}
\begin{definition}
Let $p,q \in [1,\infty]$ and $ s \in \R$. We define the Besov space $B^{s}_{p,q}(\R^d)$ and Triebel-Lizorkin space $F^{s}_{p,q}(\R^d)$ as the set of $f \in \mathcal{S}'(\R^d)$ such that the following quasinorm (with the obvious modification when $q=\infty$) is finite
\begin{equation}
\label{standard space}
\begin{split}
\norm{f}{B^{s}_{p,q}} &= \left(\sum_{j \geq -1} 2^{jsq}\norm{\Delta_jf}{L^p}^q\right)^{1/q},\\
\norm{f}{F^{s}_{p,q}} &= \norm{\left( \sum_{j \geq -1} 2^{jsq}\abs{\Delta_jf}^q\right)^{1/q}}{L^p},
\end{split}
\end{equation}
We also define the homogeneous Besov space $\dot{B}^{s}_{p,q}(\R^d)$ and homogeneous Triebel-Lizorkin space $\dot{F}^{s}_{p,q}(\R^d)$ as the set of $f\in \mathcal{S}'(\R^d)/\mathcal{P}(\R^d)$ such that the following quasinorm (with the obvious modification when $q=\infty$) is finite
\begin{equation}
\label{standard space}
\begin{split}
\norm{f}{\dot{B}^{s}_{p,q}} &= \left(\sum_{j \in \mathbb{Z}} 2^{jsq}\norm{\dot{\Delta}_jf}{L^p}^q\right)^{1/q},\\
\norm{f}{\dot{F}^{s}_{p,q}} &= \norm{\left( \sum_{j \in \mathbb{Z}} 2^{jsq}\abs{\dot{\Delta}_jf}^q\right)^{1/q}}{L^p}.
\end{split}
\end{equation}
\end{definition}
 We remark that if $X^{s}_{p,q}(\R^d)$ is either $F^{s}_{p,q}(\R^d)$ or $B^{s}_{p,q}(\R^d)$, one can show $\norm{\cdot}{L^p}+\norm{\cdot}{\dot{X}^{s}_{p,q}}$ defines an equivalent norm on $X^{s}_{p,q}(\R^d)$ (see \cite{T97}). 
When working in Triebel-Lizorkin spaces, the estimates naturally lend themselves to the use of the maximal function, which we define below.

\begin{definition}
For $f \in L^1_{loc}(\R^d)$, the maximal function $\mathcal{M}f(x)$ is defined by 
$$
\mathcal{M}f(x) = \sup_{r>0}\frac{1}{|B_r(x)|}\int_{B_r(x)}|f(y)|dy.
$$
\end{definition}

\begin{lemma} \cite{FS71}
\label{vector maximal inequality}Let $(p,q) \in (1,\infty) \times (1,\infty]$ or $p=q=\infty$ be given. Suppose $\{f_j\}_{j \in \mathbb{Z}}$ is a sequence of functions in $L^p(\R^d)$ satisfying $\norm{f_j(x)}{l^q(\mathbb{Z})} \in L^p(\R^d)$. Then
\begin{equation}
\norm{\left(\sum_{j \in \mathbb{Z}} |\mathcal{M}f_j(x)|^q\right)^{\frac{1}{q}}}{L^p} \leq C \norm{\left(\sum_{j \in \mathbb{Z}} |f_j(x)|^q\right)^{\frac{1}{q}}}{L^p}.
\end{equation}
\end{lemma}

\begin{lemma}\cite{S70}
\label{convolution bound by maximal function}
Let $\varphi \in L^1(\R^d)$ and set $\varphi_{\eps}(x)=\frac{1}{\eps^d}\varphi \left(\frac{x}{\eps}\right)$ for $\eps>0$. Further assume $\varphi$ is such that, if we set
$$
\psi(x) = \sup_{|y|\geq|x|} |\varphi(y)|,
$$
then $\norm{\psi}{L^1}=A<\infty$. Then for any $f \in L^p(\R^d)$ and $1 \leq p \leq \infty$ one has
\begin{equation}
\sup_{\eps>0}|(f\ast \varphi_{\eps})(x)| \leq A\mathcal{M}f(x).
\end{equation}
\end{lemma}

In what follows, we aim to refine the scale of Besov and Triebel-Lizorkin spaces by replacing the factor $2^{js}$ in \eqref{standard space} by $2^{js}\psi(2^{j})$ where $\psi$ is a function that grows slower than any power function and that is often referred to as a slowly varying function. We state some of the properties of slowly varying functions here but refer to \cite{K30} and \cite{BGT87} for more details.

\begin{definition}
\label{slowly varying}We let $\mathcal{M}$ denote the set of all functions $\psi:[0,\infty) \to [0,\infty)$ such that 
\begin{enumerate}
\item $\psi$ is continuous and $\psi$ is bounded away from zero on every interval of the form $[a,\infty)$ where $a>0$.
\item $\psi$ is slowly varying in the Karamata sense. i.e. 
$$
\lim_{t \to \infty} \frac{\psi(\lambda t)}{\psi(t)} = 1 \text{ for each } \lambda>0.
$$
\end{enumerate}
\end{definition}

We have the following representation of slowly varying functions from \cite{K30}.
\begin{proposition}
\label{Karamata Representation}
A function $\psi$ is slowly varying in the Karamata sense if and only if 
\begin{equation}
\label{representation of weight}
\psi(t) = c(t)\exp\left(\int_a^t \frac{\eps(s)}{s}\,ds\right),
\end{equation}
where $a>0$, $c(t)$ is measurable, $c(t) \to c>0$, and $\eps(t) \to 0$ as $t \to \infty.$
\end{proposition}
Using  \eqref{representation of weight}, it is straightforward to show that if $\psi$ is slowly varying, then
$$
\lim_{t \to \infty} \frac{\psi(t)}{t^{\beta}}=0 \text{ for any } \beta>0.
$$



A basic example of a slowly varying function is $\psi(t)=\log^\alpha(e+t)-1$ for $\alpha >0$, which satisfies \eqref{representation of weight} with $\eps(s)=\frac{\alpha}{\log(e+s)}, a=e-1,$ and $c=1$. Surprisingly, one can construct a slowly varying function $\psi$ that experiences infinite oscillation in the sense that $\liminf_{x \to \infty}\psi(x)=0$ and $\limsup_{x \to \infty}\psi(x)=\infty$. The construction of such a function can be found in \cite{BKS13}.

\begin{definition}
\label{Generalized function spaces}
Let $p,q \in [1,\infty], s \in \R$, and $ \psi \in \mathcal{M}$. We define the Besov and Triebel-Lizorkin spaces of generalized smoothness, $B^{s,\psi}_{p,q}(\R^d)$ and $F^{s,\psi}_{p,q}(\R^d)$, respectively,  as the set $f \in \mathcal{S}'(\R^d)$ such that the following quasinorms are finite
\begin{equation}
\label{generalized function spaces}
\begin{split}
\norm{f}{B^{s,\psi}_{p,q}} &= \left(\sum_{j \geq -1} 2^{jsq}\psi^q(2^j)\norm{\Delta_jf}{L^p}^q\right)^{1/q},\\
\norm{f}{F^{s,\psi}_{p,q}} &= \norm{\left( \sum_{j \geq -1} 2^{jsq}\psi^q(2^j)\abs{\Delta_jf}^q\right)^{1/q}}{L^p},
\end{split}
\end{equation}
with the obvious modifications when $q=\infty$.
We also define the homogeneous spaces $\dot{B}^{s,\psi}_{p,q}(\R^d)$ and $\dot{F}^{s,\psi}_{p,q}(\R^d)$ as the set of $f \in \mathcal{S}'(\R^d) /\mathcal{P}(\R^d)$ such that the following quasinorms are finite
\begin{equation}
\label{homogeneous generalized function spaces}
\begin{split}
\norm{f}{\dot{B}^{s,\psi}_{p,q}} &= \left(\sum_{j \in \mathbb{Z}} 2^{jsq}\psi^q(2^j)\norm{\dot{\Delta}_jf}{L^p}^q\right)^{1/q},\\
\norm{f}{\dot{F}^{s,\psi}_{p,q}} &= \norm{\left( \sum_{j \in \mathbb{Z}} 2^{jsq}\psi^q(2^j)\abs{\dot{\Delta}_jf}^q\right)^{1/q}}{L^p},
\end{split}
\end{equation}
with the obvious modifications when $q=\infty$.
\end{definition}
For properties of these spaces, such as completeness and independence of the dyadic partition of unity, we refer the reader to \cite{Moura01}. We remark that if $X^{s,\psi}_{p,q}(\R^d)$ is either $F^{s,\psi}_{p,q}(\R^d)$ or $B^{s,\psi}_{p,q}(\R^d)$, one can show $\norm{\cdot}{L^p}+\norm{\cdot}{\dot{X}^{s,\psi}_{p,q}}$ defines an equivalent norm on $X^{s,\psi}_{p,q}(\R^d)$.

\begin{proposition}
\label{embedding theorem}
Suppose that $s\geq d/p$,  $ \psi \in \mathcal{M}$, and $\frac{1}{q}+\frac{1}{q'}=1$. If
$$
\int_{1}^{\infty} \frac{dt}{t\psi^{q'}(t)}<\infty,
$$
then $B^{s,\psi}_{p,q}(\R^d) \hookrightarrow B^{s-d/p}_{\infty,1}(\R^d)$.
\end{proposition}
\begin{proof}
It suffices to show that $B^{s,\psi}_{p,q}(\R^d)\hookrightarrow B^{s}_{p,1}(\R^d)$ and appeal to the classical embeddings of Besov spaces. Indeed, note that
$$
\sum_{j \geq -1} 2^{js}\norm{\Delta_j f}{L^p} \leq \left(\sum_{j \geq -1} \frac{1}{\psi^{q'}(2^j)}\right)^{\frac{1}{q'}} \norm{f}{B^{s,\psi}_{p,q}}.
$$
Then,
$$
\sum_{j \geq -1} \frac{1}{\psi^{q'}(2^j)} \sim \int_\frac{1}{2}^{\infty} \frac{1}{\psi^{q'}(2^x)} \,dx \sim  \int_1^{\infty} \frac{1}{t\psi^{q'}(t)}\, dt,
$$
where in the last inequality we took the substitution $x=\log_2(t)$.
\end{proof}

\begin{proposition}
\label{key triebel lizorkin embedding}
Suppose that $s\geq d/p$.  $ \psi \in \mathcal{M}$, and $\frac{1}{p}+\frac{1}{p'}=1$. Then $F^{s,\psi}_{p,q}(\R^d) \hookrightarrow B^{s-\frac{d}{p}}_{\infty,1}(\R^d)$ if
$$
\int_1^{\infty} \frac{1}{t\psi^{p'}(t)}\,dt<\infty.
$$
\end{proposition}
\begin{proof}
We begin by recalling that if $sp\geq d$, then  $F^{s}_{p,q}(\R^d) \hookrightarrow B^{s-\frac{d}{p}}_{\infty,p}(\R^d)$ (see 11.4(iii) in \cite{T97}). A similar argument yields $F^{s,\psi}_{p,q}(\R^d) \hookrightarrow B^{s-\frac{d}{p},\psi}_{p,q}(\R^d)$. Then, by Hölder's inequality, $B^{s-\frac{d}{p},\psi}_{p,q}(\R^d) \hookrightarrow B^{s-\frac{d}{p}}_{\infty,1}(\R^d)$ if $((\psi(2^j))^{-1}) \in l^{p'}(\mathbb{Z}_{\geq-1})$ which, as in the proof above, is equivalent to the condition that $\int_1^{\infty} \frac{1}{t\psi^{p'}(t)}\,dt<\infty$.
\end{proof}

\begin{definition}
\label{good sv functions} Let $r \in [1,\infty].$ We say that $\psi \in \mathcal{M}_{r}$ if
\begin{enumerate}
    \item $\psi \in \mathcal{M}$.
    \item $\psi$ is non-decreasing.
    \item $\int_1^{\infty} \frac{1}{t\psi^r(t)}\,dt<\infty.$
\end{enumerate}
\end{definition}

A basic example of such a function is $\psi(t)=\log^{\alpha}(e+t)-1$ for $\alpha > \frac{1}{r}.$ We do not consider oscillatory slowly varying functions since we are only concerned about the local smoothness of the functions in these generalized spaces, which is parametrized by the growth rate of $\psi$ at infinity. Furthermore, when proving calculus estimates in the appendix, we lean heavily on the fact that $\psi$ is non-decreasing. If $\psi$ were to have a monotone equivalence, the results would still follow. But, in general, not every slowly varying function has a monotone equivalence (see e.g. \cite{BKS13}).

\section{Well-Posedness in Generalized Besov Spaces}\label{Section Besov}
\subsection{Short-Time Existence}

We now prove an a priori estimate of smooth solutions to \eqref{euler velocity} in $B^{s,\psi}_{p,q}(\R^d)$.
\begin{theorem}\label{euler a-priori estimate besov} Suppose $u$ is a smooth solution to \eqref{euler velocity} on $[0,T]\times \R^d$ with $u_0 \in B^{s,\psi}_{p,q}(\R^d)$, $s=d/p+1$, $p\in(1,\infty)$, and $q \in [1,\infty]$. Then for any $t \leq T$
\begin{equation}
\norm{u(t)}{B^{s,\psi}_{p,q}}\leq \norm{u_0}{B^{s,\psi}_{p,q}}\exp\left(C\int_0^t \norm{\nabla u(\tau)}{L^{\infty}}\,d\tau\right).
\end{equation}
\end{theorem}
\begin{proof}
We begin by applying $\Delta_j$ to both sides of $\eqref{euler velocity}_1$ to see 
\begin{equation}
\label{besov apriori pre langrian}
\partial_t \Delta_j u + u \cdot \nabla \Delta_j u = [u \cdot \nabla, \Delta_j]u  - \Delta_j\nabla p,
\end{equation}
where $[u \cdot \nabla, \Delta_j]u=u\cdot \nabla \Delta_ju - \Delta_j(u\cdot \nabla u)$. Let $X_t(x)$ be the flow map the solving the ordinary differential equation
\begin{equation}
\label{flow map ode}
\begin{cases}
\frac{d}{dt}X_t(x) = u(t,X_t(x)),\\
X_t(x)\vert_{t=0}=x.
\end{cases}
\end{equation}
Using the chain rule we rewrite \eqref{besov apriori pre langrian} in Lagrangrian coordinates to see
\begin{equation}
\label{pre L^p apriori besov 0}
\frac{d}{dt}(\Delta_ju(t,X_t(x))) = [u \cdot \nabla, \Delta_j]u(t,X_t(x))-\Delta_j\nabla p (t,X_t(x)).
\end{equation}
Applying $\tilde{\Delta}_j$ to both sides of \eqref{pre L^p apriori besov 0} and using that $\Delta_j=\Delta_j\tilde{\Delta}_j$ yields
\begin{equation}
\label{pre L^p apriori besov}
\frac{d}{dt}(\Delta_ju(t,X_t(x))) = \tilde{\Delta}_j[u \cdot \nabla, \Delta_j]u(t,X_t(x))-\Delta_j\nabla p (t,X_t(x)).
\end{equation}
Integrating \eqref{pre L^p apriori besov}  in time, taking the $L^p(\mathbb{R}^d)$ norm, and using that $X_t$ is a volume preserving diffeomorphism due to the divergence-free condition on $u$ gives
\begin{equation}
\label{L^p apriori besov}
\norm{\Delta_j u(t)}{L^p} \leq \norm{\Delta_ju_0}{L^p} + \int_0^t \left\{ \norm{\tilde{\Delta}_j[u \cdot \nabla, \Delta_j]u(\tau)}{L^p} + \norm{\Delta_j\nabla p(\tau)}{L^p}\right\}\,d\tau.
\end{equation}
Multiplying by $2^{js}\psi(2^j)$, taking the $l^q(\mathbb{Z}_{\geq -1})$ norm in $j$, and using Minkowski's inequality yields
\begin{equation}
\label{gbesov estimate 1}
\norm{u(t)}{B^{s,\psi}_{p,q}} \leq \norm{u_0}{B^{s,\psi}_{p,q}} + \int_0^t \norm{\nabla p(\tau)}{B^{s,\psi}_{p,q}}\,d\tau + \int_0^t \left\{ \sum_{j \geq -1} \left(2^{js}\psi(2^j)\norm{\tilde{\Delta}_j[u \cdot \nabla, \Delta_j]u(\tau)}{L^p}\right)^q\right\}^{\frac{1}{q}}\,d\tau.
\end{equation}
By Proposition \ref{besov commutator estimate} we see 
\begin{equation}
\label{besov commutator apriori}
\int_0^t \left\{ \sum_{j \geq -1} \left(2^{js}\psi(2^j)\norm{\tilde{\Delta}_j[u \cdot \nabla, \Delta_j]u(\tau)}{L^p}\right)^q\right\}^{\frac{1}{q}}\,d\tau \lesssim \int_0^t\norm{\nabla u(\tau)}{L^{\infty}}\norm{u(\tau)}{B^{s,\psi}_{p,q}}\,d\tau.
\end{equation}
We next estimate the pressure. Recall 
$$
p=\sum_{i,j=1}^d (-\Delta)^{-1}\partial_iu_j\partial_ju_i.
$$
Estimating the homogeneous norm first gives
\begin{align*}
\norm{\nabla p}{\dot{B}^{s,\psi}_{p,q}}&\leq \sum_{i,j=1}^d \norm{\nabla (-\Delta)^{-1}\partial_iu_j\partial_ju_i}{\dot{B}^{s,\psi}_{p,q}}\\
&
\lesssim \sum_{i,j=1}^d \norm{\partial_iu_j\partial_ju_i}{\dot{B}^{s-1,\psi}_{p,q}}\\
&\lesssim \norm{\nabla u}{L^{\infty}}\norm{u}{\dot{B}^{s,\psi}_{p,q}},
\end{align*}
where in the last inequality we used Lemma \ref{Leibniz Rule in Besov Spaces}.
Combining this with the estimate
$$
\norm{\nabla p}{L^p} \lesssim \norm{(u \cdot \nabla) u}{L^p} \lesssim \norm{\nabla u}{L^{\infty}}\norm{u}{L^p},
$$
allows us to conclude 
\begin{equation}
\label{Pressure apriori besov estimate}
\norm{\nabla p}{B^{s,\psi}_{p,q}} \lesssim \norm{\nabla u}{L^{\infty}}\norm{u}{B^{s,\psi}_{p,q}}.
\end{equation}
Substituting \eqref{besov commutator apriori} and \eqref{Pressure apriori besov estimate} into \eqref{gbesov estimate 1} gives
\begin{equation}
\label{gbesov estimate 2}
\norm{u(t)}{B^{s,\psi}_{p,q}}\leq \norm{u_0}{B^{s,\psi}_{p,q}}+C\int_0^t \norm{\nabla u(\tau)}{L^{\infty}}\norm{u(\tau)}{B^{s,\psi}_{p,q}}\,d\tau.
\end{equation}
An application of Grönwall's Lemma yields the desired result.
\end{proof}

\begin{theorem}\label{short time euler besov} Suppose that $u_0 \in B^{s,\psi}_{p,q}(\R^d)$, $d\geq 2$, $s=d/p+1$, $p \in (1,\infty), q\in [1,\infty]$, and $\psi \in \mathcal{M}_{q'}$. Then, there exists a $T=T(\norm{u_0}{B^{s,\psi}_{p,q}})>0$ such that a unique solution
$$
u \in C([0,T]; B^{s,\psi}_{p,q}(\R^d)),
$$
of the system \eqref{euler velocity} exists.
\end{theorem}
\begin{proof}
\textbf{Defining an approximating sequence:} We begin by the defining the sequence $(u^{(n)},p^{(n)})$ as follows. We solve the linearized problem
\begin{equation}
\label{approximating sequence PDE}
\begin{cases}
\partial_t u^{(n)}+(u^{(n-1)}\cdot \nabla)u^{(n)}=-\nabla p^{(n-1)},\\
\text{div } u^{(n)}=0,\\
u^{(n)}\vert_{t=0}=S_nu_0,\\
\end{cases}
\end{equation}
for $n \in \mathbb{N}$ where,
\begin{equation}
\label{pressure approximation}
p^{(n)}=\sum_{i,j=1}^d (-\Delta)^{-1}(\partial_iu_j^{(n)}\partial_ju_i^{(n)}), \hspace{.5cm} n  \in \mathbb{N},
\end{equation}
and $u^{(0)}=p^{(0)}=0$.
\\
\\
\textbf{Uniform bounds:} We let $X^s_T= C([0,T];B^{s,\psi}_{p,q}(\R^d))$ with $\|f\|_{X^s_T}:=\sup_{t\in[0,T]}\|f(t,\cdot)\|_{B^{s,\psi}_{p,q}}$. Using induction we aim to show
\begin{equation}
\label{uniform bound besov}
\norm{u^{(n)}}{X_T^s}\leq 2\norm{u_0}{B^{s,\psi}_{p,q}} \text{ for all } n\geq 0,
\end{equation}
when $T>0$ is sufficiently small. Following the proof of Theorem \ref{euler a-priori estimate besov} and utilizing the continuous embedding $B^{s,\psi}_{p,q}(\R^d)\hookrightarrow B^{1}_{\infty,1}(\R^d)$ for $\psi \in \mathcal{M}_{q'}$ we see the solution $(u^{(n)},p^{(n)})$ to \eqref{approximating sequence PDE} satisfies
\begin{equation}
\label{besov induction eq 1}
\norm{u^{(n)}(t)}{B^{s,\psi}_{p,q}} \leq C\left(\norm{u_0}{B^{s,\psi}_{p,q}} + \int_0^t \left\{ \norm{u^{(n-1)}(\tau)}{B^{s,\psi}_{p,q}}^2 + \norm{u^{(n-1)}(\tau)}{B^{s,\psi}_{p,q}}\norm{u^{(n)}(\tau)}{B^{s,\psi}_{p,q}}\right\}\,d\tau\right).
\end{equation}
The base case $n=0$ follows immediately from the definition of the approximation sequence. We then assume \eqref{uniform bound besov} holds for $n=1,2,\hdots,k-1$. Using \eqref{besov induction eq 1} and the inductive hypothesis we have
\begin{equation}
\label{besov induction eq 2}
\norm{u^{(k)}(t)}{B^{s,\psi}_{p,q}} \leq C\norm{u_0}{B^{s,\psi}_{p,q}}\left(1 + 4T\norm{u_0}{B^{s,\psi}_{p,q}} + 2\int_0^t \norm{u^{(k)}(\tau)}{B^{s,\psi}_{p,q}}\,d\tau\right).
\end{equation}
Then, an application of Grönwall's Lemma to \eqref{besov induction eq 2} gives
$$
\norm{u^{(k)}}{X^s_T} \leq \norm{u_0}{B^{s,\psi}_{p,q}}C\left(1+4T\norm{u_0}{B^{s,\psi}_{p,q}}\right)\exp\left(2T\norm{u_0}{B^{s,\psi}_{p,q}}\right).
$$
We then choose $T\leq T_0$ where
\begin{equation}
\label{Short time T}
T_0= \min \left\{ \frac{5-4C}{16C\norm{u_0}{B^{s,\psi}_{p,q}}}, \frac{\ln \left(5/4\right)}{2\norm{u_0}{B^{s,\psi}_{p,q}}}\right\},
\end{equation}
so that
$$C\left(1+4T\norm{u_0}{B^{s,\psi}_{p,q}}\right)\exp\left(2T\norm{u_0}{B^{s,\psi}_{p,q}}\right)\ \leq \left(\frac{5}{4}\right)^2 \leq 2,
$$
establishing \eqref{uniform bound besov}.\\ \\
\textbf{$u^{(n)}$ is Cauchy in $X^{s-1}_T$:} Without loss of generality, assume that $n \geq m$. We see the difference of $u^{(n)}$ and $u^{(m)}$ must satisfy the following differential equation
\begin{equation}
\label{difference eq 1 besov}
\begin{cases}
\partial_t(u^{(n)}-u^{(m)})+(u^{(n-1)}\cdot\nabla)(u^{(n)}-u^{(m)}) = - ((u^{(n-1)}-u^{(m-1)})\cdot\nabla)u^{(m)} - \nabla (p^{(n-1)}-p^{(m-1)}),\\
\left(u^{(n)}-u^{(m)}\right)\vert_{t=0}=\sum_{j=1}^{n-m}\Delta_{j+m}u_0.
\end{cases}
\end{equation}
In what follows, we let $X_{t,n-1}$ be the corresponding flow map for $u^{(n-1)}$. Applying $\dot{\Delta}_j$ to \eqref{difference eq 1 besov}$_1$ and writing the result in Lagrangian coordinates yields
\begin{multline}
\label{difference eq 2 besov}
\frac{d}{dt}\left( \dot{\Delta}_j(u^{(n)}-u^{(m)})(t,X_{t,n-1}(x))\right) = -\dot{\Delta}_j\left(((u^{(n-1)}-u^{(m-1)})\cdot\nabla)u^{(m)}\right)(t,X_{t,n-1}(x))\\ - \dot{\Delta}_j\nabla(p^{(n-1)}-p^{(m-1)})(t,X_{t,n-1}(x)) + \tilde{\Delta}_j\left(\left([u^{(n-1)}\cdot\nabla,\dot{\Delta}_j](u^{(n)}-u^{(m)})\right)(t,X_{t,n-1}(x))\right).
\end{multline}
We integrate \eqref{difference eq 2 besov} in time, take the $L^p$ norm on both sides, and use that $X_{t,n-1}$ is measure preserving to obtain
\begin{multline}
\label{difference eq 3 besov}
\norm{\dot{\Delta}_j(u^{(n)}-u^{(m)})(t)}{L^p}\leq \norm{\dot{\Delta}_j\left(\sum_{j=1}^{n-m} \Delta_{j+m}u_0\right)}{L^p} + \int_0^t\norm{\dot{\Delta}_j(((u^{(n-1)}-u^{(m-1)})\cdot \nabla)u^{(m)})(\tau)}{L^p}\,d\tau\\
+\int_0^t \norm{\dot{\Delta}_j\nabla(p^{(n-1)}-p^{(m-1)})(\tau)}{L^p} + \int_0^t \norm{\tilde{\Delta}_j\left([u^{(n-1)}\cdot\nabla,\dot{\Delta}_j](u^{(n)}-u^{(m)})\right)(\tau)}{L^p}\,d\tau.
\end{multline}
We multiply both sides of \eqref{difference eq 3 besov} by $2^{j(s-1)}\psi(2^j)$, take the $l^q(\mathbb{Z})$ norm, and use Minkowski's inequality to obtain
\begin{multline}
\label{difference eq 4 besov}
\norm{(u^{(n)}-u^{(m)})(t)}{\dot{B}^{s-1,\psi}_{p,q}} \leq \norm{\sum_{j=1}^{n-m}\Delta_{j+m}u_0}{\dot{B}^{s-1,\psi}_{p,q}} + \int_0^t \norm{((u^{(n-1)}-u^{(m-1)})\cdot \nabla)u^{(m)}(\tau)}{\dot{B}^{s-1,\psi}_{p,q}}\,d\tau\\
 + \int_0^t \norm{2^{j(s-1)}\psi(2^j)\norm{\tilde{\Delta}_j\left([u^{(n-1)}\cdot\nabla,\dot{\Delta}_j](u^{(n)}-u^{(m)})\right)(\tau)}{L^p}}{l^q(\mathbb{Z})}\,d\tau\\
+\int_0^t \norm{\nabla(p^{(n-1)}-p^{(m-1)})(\tau)}{\dot{B}^{s-1,\psi}_{p,q}}\,d\tau
\eqdef I + II + III + IV.
\end{multline}
\textbf{Estimate of term $I$:} We have 
\begin{align*}
I =& \norm{\sum_{j=1}^{n-m}\Delta_{j+m}u_0}{\dot{B}^{s-1,\psi}_{p,q}}\leq \sum_{j=1}^{n-m} \norm{\Delta_{j+m}u_0}{B^{s-1,\psi}_{p,q}}\\
&\lesssim \sum_{j=1}^{n-m} 2^{-(j+m)}\norm{\Delta_{j+m}u_0}{B^{s,\psi}_{p,q}} \lesssim 2^{-m}\norm{u_0}{B^{s,\psi}_{p,q}}\sum_{j=1}^{n-m}2^{-j}\\ &\leq 2^{-m}\norm{u_0}{B^{s,\psi}_{p,q}}.
\end{align*}
\textbf{Estimate of term $II$:} By Proposition \ref{Leibniz Rule in Besov Spaces} and Proposition \ref{embedding theorem}, we have
\begin{align*}
II &= \int_0^t \norm{((u^{(n-1)}-u^{(m-1)})\cdot \nabla)u^{(m)}(\tau)}{\dot{B}^{s-1,\psi}_{p,q}}d \tau\\ &\lesssim \int_0^t \norm{(u^{(n-1)}-u^{(m-1)})(\tau)}{\dot{B}^{s-1,\psi}_{p,q}} \norm{\nabla u^{(m)}(\tau)}{L^{\infty}}\,d\tau \\
& \hspace{1cm}+ \int_0^t \norm{\nabla u^{(m)}(\tau)}{\dot{B}^{s-1,\psi}_{p,q}}\norm{(u^{(n-1)} -u^{(m-1)})(\tau)}{L^{\infty}}\,d\tau\\
&\lesssim \int_0^t \norm{u^{(m)}(\tau)}{B^{s,\psi}_{p,q}}\norm{(u^{(n-1)}-u^{(m-1)})(\tau)}{B^{s-1,\psi}_{p,q}}\,d\tau.
\end{align*}
\textbf{Estimate of term $III$:} By Proposition \ref{besov commutator estimate} and Proposition \ref{embedding theorem}, we have 
\begin{align*}
III &= \int_0^t \norm{2^{j(s-1)}\psi(2^j)\norm{\tilde{\Delta}_j\left([u^{(n-1)}\cdot\nabla,\dot{\Delta}_j](u^{(n)}-u^{(m)})\right)(\tau)}{L^p}}{l^q(\mathbb{Z})}\,d\tau\\
&\lesssim
\int_0^t \norm{\nabla u^{(n-1)}(\tau)}{L^{\infty}}\norm{(u^{(n)}-u^{(m)})(\tau)}{\dot{B}^{s-1,\psi}_{p,q}} \,d\tau \\
&\hspace{1cm}+ \int_0^t \norm{ (u^{(n)}-u^{(m)})(\tau)}{L^{\infty}}\norm{\nabla u^{(n-1)}(\tau)}{\dot{B}^{s-1,\psi}_{p,q}}\,d\tau\\
&\lesssim \int_0^t \norm{u^{(n-1)}(\tau)}{B^{s,\psi}_{p,q}}\norm{(u^{(n)}-u^{(m)})(\tau)}{B^{s-1,\psi}_{p,q}}\,d\tau.
\end{align*}
\textbf{Estimate of $IV$:} We first come up with an useful expression for $\nabla p^{(n-1)}-\nabla p^{(m-1)}$. By \eqref{pressure approximation}, we see that
\begin{align*}
&\nabla p^{(n-1)}-\nabla p^{(m-1)} 
= 
\sum_{i,j=1}^d (-\Delta)^{-1}\nabla (\partial_iu_j^{(n-1)}\partial_j u_i^{(n-1)} - \partial_i u_j^{(m-1)}\partial_ju_i^{(m-1)})\\
&=
\sum_{i,j=1}^d (-\Delta)^{-1}\nabla\partial_i\left(u_j^{(n-1)}\partial_ju_i^{(n-1)}- u_j^{(m-1)}\partial_ju^{(m-1)}\right) \\
&\hspace{.5cm} - \sum_{i,j=1}^d (-\Delta)^{-1}\nabla\left(u_j^{(n-1)}\partial_i\partial_ju_i^{(n-1)}-u_j^{(m-1)}\partial_i\partial_ju_i^{(m-1)}\right).
\end{align*}
Note that since $u^{(k)}$ is divergence-free for all $k \in \mathbb{N}$ one has
\begin{align*}
&\sum_{i,j=1}^d (-\Delta)^{-1}\nabla\left(u_j^{(n-1)}\partial_i\partial_ju_i^{(n-1)}-u_j^{(m-1)}\partial_i\partial_ju_i^{(m-1)}\right) 
\\
&=
\sum_{j=1}^d (-\Delta)^{-1}\nabla\partial_j\left\{ u_j^{(n-1)}\sum_{i=1}^d \partial_iu_i^{(n-1)} - u_j^{(m-1)} \sum_{i=1}^d \partial_i u_i^{(m-1)} \right\}=0.
\end{align*}
We then continue rewriting the difference of the pressure terms as 
\begin{align*}
&\nabla p^{(n-1)}-\nabla p^{(m-1)} 
= 
\sum_{i,j=1}^d (-\Delta)^{-1}\nabla\partial_i\left(u_j^{(n-1)}\partial_ju_i^{(n-1)}- u_j^{(m-1)}\partial_ju^{(m-1)}\right) \\
&=
\sum_{i,j=1}^d(-\Delta)^{-1}\nabla \partial_i\left(u_j^{(n-1)}\partial_ju_i^{(n-1)}-u_j^{(n-1)}\partial_ju_i^{(m-1)}+u_j^{(n-1)}\partial_ju_i^{(m-1)}- u_j^{(m-1)}\partial_ju_i^{(m-1)}\right)\\
&=
\sum_{i,j=1}^d(-\Delta)^{-1}\nabla \partial_i\left(u_j^{(n-1)}\partial_j\left(u_i^{(n-1)}-u_i^{(m-1)}\right)+\partial_ju_i^{(m-1)}\left(u_j^{(n-1)}- u_j^{(m-1)}\right)\right)\\
&= 
\sum_{i,j=1}^d (-\Delta)^{-1}\nabla \partial_j\left\{\left(\partial_iu_j^{(n-1)}\right)\left(u_i^{(n-1)}-u_i^{(m-1)}\right)\right\}\\ 
&\hspace{1cm}+ \sum_{i,j=1}^d (-\Delta)^{-1}\nabla \partial_i \left\{\left(\partial_ju_i^{(m-1)}\right)\left(u_j^{(n-1)}-u_j^{(m-1)}\right)\right\},
\end{align*}
where in the last equality we used the incompressibility condition of the velocity field. Since $(-\Delta)^{-1}\nabla\partial_j$ and $(-\Delta)^{-1}\nabla\partial_i$ are Calder\'{o}n-Zygmund operators we have
\begin{align*}
&\int_0^t \norm{\nabla(p^{(n-1)}-p^{(m-1)})(\tau)}{\dot{B}^{s-1,\psi}_{p,q}}\,d\tau \lesssim \int_0^t \norm{((u^{(n-1)}-u^{(m-1)})\cdot\nabla)u^{(n-1)}(\tau)}{\dot{B}^{s-1,\psi}_{p,q}}\,d\tau\\ &\hspace{5cm}+ \int_0^t \norm{((u^{(n-1)}-u^{(m-1)})\cdot \nabla) u^{(m-1)}(\tau)}{\dot{B}^{s-1,\psi}_{p,q}}\,d\tau\\
&\lesssim \int_0^t \left( \norm{u^{(n-1)}(\tau)}{B^{s,\psi}_{p,q}} + \norm{u^{(m-1)}(\tau)}{B^{s,\psi}_{p,q}}\right)\norm{(u^{(n-1)}-u^{(m-1)})(\tau)}{\dot{B}^{s-1,\psi}_{p,q}}\,d\tau,
\end{align*}
where in the last inequality we applied Proposition \ref{Leibniz Rule in Besov Spaces}.
\\
\\
Combining the estimates for $I,II,III,$ and $IV$ we have 
\begin{multline}
\label{short time cauchy homogeneuous besov}
\norm{(u^{(n)}-u^{(m)})(t)}{\dot{B}^{s-1,\psi}_{p,q}} \lesssim 2^{-m}\norm{u_0}{B^{s,\psi}_{p,q}}  +\int_0^t \norm{u^{(n-1)}(\tau)}{B^{s,\psi}_{p,q}}\norm{(u^{(n)}-u^{(m)})(\tau)}{B^{s-1,\psi}_{p,q}}\,d\tau\\
+\int_0^t \left(\norm{u^{(m)}(\tau)}{B^{s,\psi}_{p,q}} + \norm{u^{(n-1)}(\tau)}{B^{s,\psi}_{p,q}} + \norm{u^{(m-1)}(\tau)}{B^{s,\psi}_{p,q}}\right)\norm{(u^{(n-1)}-u^{(m-1)})(\tau)}{B^{s-1,\psi}_{p,q}}\,d\tau.\\
\end{multline}
We next derive the non-homogeneous estimate. Integrating \eqref{difference eq 1 besov} in time and taking the $L^p(\R^d)$ norm leads to 
\begin{multline}
\label{short time cauchy nonhomogenuous besov}
\norm{(u^{(n)}-u^{(m)})(t)}{L^p} \lesssim 2^{-m}\norm{\nabla u_0}{L^p}\\ + \int_0^t \left\{ \norm{((u^{(n-1)}-u^{(m-1)})\cdot \nabla) u^{(m)}(\tau)}{L^p} + \norm{\nabla(p^{(n-1)}-p^{(m-1)})(\tau)}{L^p}\right\}\,d\tau\\
\lesssim  2^{-m}\norm{\nabla u_0}{L^p} + \int_0^t \norm{(u^{(n-1)}-u^{(m-1)})(\tau)}{L^p}\norm{\nabla u^{(m)}(\tau)}{L^{\infty}}\,d\tau\\ + \int_0^t \left(\norm{\nabla u^{(n-1)}(\tau)}{L^{\infty}}+\norm{\nabla u^{(m-1)}(\tau)}{L^{\infty}}\right)\norm{(u^{(n-1)}-u^{(m-1)})(\tau)}{L^p}\,d\tau.
\end{multline}
Adding \eqref{short time cauchy nonhomogenuous besov} to \eqref{short time cauchy homogeneuous besov} we find that there exists some constant $C_1=C_1(\norm{u_0}{B^{s,\psi}_{p,q}})$ such that 
$$
\norm{u^{(n)}-u^{(m)}}{X_T^{s-1}} \leq C_12^{-m} + C_1T\norm{u^{(n-1)}-u^{(m-1)}}{X^{s-1}_T} + C_1T\norm{u^{(n)}-u^{(m)}}{X_T^{s-1}}.
$$
Choosing $T\leq \min\left\{T_0, \frac{1}{2C_1}\right\}$ we find 
$$
\norm{u^{(n)}-u^{(m)}}{X^{s-1}_T}\leq C_12^{-m+1} +  2C_1T\norm{u^{(n-1)}-u^{(m-1)}}{X^{s-1}_T}.
$$
Iterating this inequality one obtains
\begin{equation}
\label{besov cauchy inequality 1}
\norm{u^{(n)}-u^{(m)}}{X^{s-1}_T} \leq C_1\left(\sum_{k=0}^m 2^{-(m-k)+1}(2C_1T)^k\right) + C(2C_1T)^m\norm{u_0}{X^{s-1}_T}.
\end{equation}
For the sum in \eqref{besov cauchy inequality 1} we see
$$
\sum_{k=0}^m 2^{-(m-k)+1}(2C_1T)^k=2^{1-m}\sum_{k=0}^m 2^{2k}(C_1T)^k=2^{1-m}\sum_{k=0}^{m}(4C_1T)^k.
$$
Thus, we set $T \leq \min\{T_0,\frac{1}{4C_1}\}$ and conclude that $u^{(n)}$ is a Cauchy sequence in $X^{s-1}_T$. So there exists some $u \in X^{s-1}_T$ such that $u^{(n)} \to u$ in $X^{s-1}_T$.
\\
\\
\textbf{$u$ satisfies (\ref{euler velocity})}
We show now that the limit $u$ is a weak solution to \eqref{euler velocity} in the sense of definition \ref{weak solution}. Since each $u^{(n)}$ is a classical solution to \eqref{approximating sequence PDE} we see 
\begin{equation}
\label{weak solution eq 1}
\int_0^T \int_{\R^d}\left\{ u^{(n)}\partial_t\phi-u_k^{(n)}u_{l}^{(n-1)}\partial_{x_l}\phi_k\right\}\,dx\,dt = \int_{\R^d} \left\{u^{(n)}(T)\cdot\phi(T)-S_nu_0\cdot\phi(0)\right\}\,dx
\end{equation}
for all $\phi \in C_0^{\infty}([0,T]; \R^d)$. Since $u_0 \in B^{s,\psi}_{p,q}(\R^d)\hookrightarrow B_{\infty,1}^1(\R^d)$ we see that $S_nu_0 \to u_0$ in $L^{\infty}(\R^d)$. Then by the dominated convergence theorem 
\begin{equation}
\label{weak soln 1}
\lim_{n \to \infty} \int_{\R^d} S_nu_0 \cdot \phi(0)\,dx = \int_{\R^d} u_0 \cdot \phi(0)\,dx.
\end{equation}
By the uniform bound \eqref{uniform bound besov} and uniqueness of weak limits, we see that $u_n \overset{\ast}{\rightharpoonup} u$ in $C([0,T];B^{s,\psi}_{p,q}(\R^d))$. It immediately follows that
\begin{equation}
\label{weak soln 2}
\lim_{n \to \infty} \int_0^T \int_{\R^d} u^{(n)}\partial_t\phi\, dx\,dt = \int_0^T\int_{\R^d} u\partial_t\phi\, dx\,dt \text{ and }\quad \lim_{n\to\infty}\int_{\mathbb{R}^d}u^{(n)}(T)\phi(T)\,dx=\int_{\mathbb{R}^d}u(T)\phi(T)\,dx.
\end{equation}
We now estimate the nonlinear term as
\begin{align*}
&\left| \int_0^T \int_{\R^d}u_k^{(n)}u_l^{(n-1)}\partial_{x_l}\phi_k \,dx\,dt - \int_0^T\int_{\R^d}u_ku_l\partial_{x_l}\phi_l \,dx\, dt \right|
\\
& \leq \int_0^T \int_{\R^d} |u_k^{(n)}(u_l^{(n-1)}-u_l)\partial_{x_l}\phi_k|\,dx\,dt + \int_0^T\int_{\R^d}|(u_k^{(n)}-u_k)u_l\partial_{x_l}\phi_k|\,dx\,dt\\
&\leq C\norm{u_0}{B^{s,\psi}_{p,q}}\int_0^T\int_{\R^d} | \partial_{x_l}\phi_k|\,dx\,dt \left( \norm{u_l^{(n-1)}-u_l}{X^{s-1}_T} + \norm{u_k^{(n)}-u_k}{X^{s-1}_T}\right).
\end{align*}
Since $u^{(n)}$ is Cauchy in $X^{s-1}_T$, it follows that
\begin{equation}
\label{weak soln 3}
\lim_{n \to \infty} \int_0^T \int_{\R^d}u_k^{(n)}u_l^{(n-1)}\partial_{x_l}\phi_k \,dx\,dt = \int_0^T\int_{\R^d}u_ku_l\partial_{x_l}\phi_l \,dx \,dt.
\end{equation}
By \eqref{weak soln 1}, \eqref{weak soln 2}, and \eqref{weak soln 3} we see that the limit $u \in C([0,T];B^{s,\psi}_{p,q}(\R^d))$ is a weak solution to \eqref{euler velocity}. By the embedding $B^{s,\psi}_{p,q}(\R^d) \hookrightarrow B^1_{\infty,1}(\R^d)$, we see that $u \in C([0,T];C^1(\R^d))$. Since $u$ is sufficiently regular, it follows that $u$ is a classical solution. The argument for uniqueness is identical to showing $u^{(n)}$ is Cauchy in $X^{s-1}_T$, so we omit the proof.
\end{proof}

\subsection{Global Existence }
In this section, we prove a BKM-type inequality that allows us to extend the lifetime of our solution under a certain condition. Following this, we utilize the transport structure of the $2$D vorticity equation and an estimate from \cite{MV98} to conclude global persistence of the $B^{s,\psi}_{p,q}(\R^2)$ norm.
\begin{theorem}(Blow-up Criterion)
\label{BKM-type inequality Besov Space} Let $p \in (1,\infty), q \in [1,\infty], s=d/p+1,$ and $\varphi \in \mathcal{M}_{q'}$. Then, the local in time solution $u \in C([0,T];B^{s,\psi}_{p,q}(\R^d))$ blows up at $T^*>T$ in $B^{s,\psi}_{p,q}(\R^d)$, namely
\begin{equation}
\label{blow up besov}
\limsup_{t \to T^*}\norm{u(t)}{B^{s,\psi}_{p,q}}=\infty
\end{equation}
if and only if
\begin{equation}
\label{besov BKM criterion}
\int_0^{T^*}\norm{\omega(t)}{\dot{B}^{0}_{\infty},1}\,dt=\infty.
\end{equation}
\end{theorem}
\begin{proof}
From Theorem \ref{euler a-priori estimate besov} we have
$$
\norm{u(t)}{B^{s,\psi}_{p,q}}\leq \norm{u_0}{B^{s,\psi}_{p,q}}\exp\left(C\int_0^t \norm{\nabla u(\tau)}{L^{\infty}}\,d\tau\right).
$$
Then, making use of the embedding $B_{\infty,1}^0(\mathbb{R}^d) \hookrightarrow L^{\infty}(\mathbb{R}^d)$ and the Biot-Savart law we see 
\begin{align*}
\norm{\nabla u}{L^{\infty}}&\leq\norm{\nabla \text{div}((-\Delta)^{-1}\omega)}{B^0_{\infty,1}}\\
&\leq \norm{\Delta_{-1}\nabla \text{div}((-\Delta)^{-1}\omega)}{L^{\infty}}+\sum_{j \geq 0}\norm{\Delta_j\nabla\text{div}((-\Delta)^{-1}\omega)}{L^{\infty}}\\
&\leq C(\norm{\omega_0}{L^p}+\norm{\omega}{\dot{B}^0_{\infty,1}}).
\end{align*}
It immediately follows that 
$$
\norm{u(t)}{B^{s,\psi}_{p,q}} \leq \norm{u_0}{B^{s,\psi}_{p,q}}e^{C\norm{\omega_0}{L^p}t}\exp{\left(C\int_0^t \norm{\omega(\tau)}{\dot{B}^0_{\infty,1}} \,d\tau\right)},
$$
which establishes the ``only if" part of the theorem. To see the ``if" part of the theorem note 
$$
\int_0^T \norm{\omega(\tau)}{\dot{B}^0_{\infty,1}}\,d\tau \leq T\sup_{0\leq t \leq T}\norm{\omega(t)}{\dot{B}^0_{\infty,1}} \leq CT\sup_{0 \leq t \leq T}\norm{u(t)}{B^{s,\psi}_{p,q}}
$$
which completes the proof of the theorem.
\end{proof}

\begin{corollary}
\label{global well posedness in Besov Spaces}
Let $p \in (1,\infty), q \in [1,\infty], \psi \in \mathcal{M}_{q'}$, and $s=2/p+1$. Given $u_0 \in B^{s,\psi}_{p,q}(\R^2)$ there is a unique solution to \eqref{euler velocity} with $u \in C([0,\infty);B^{s,\psi}_{p,q}(\R^2))$.
\end{corollary}
\begin{proof}
Recall that in two dimensions the vorticity $\omega$ is transported by the velocity field $u$. That is, the solution to \eqref{euler 2d vorticity} is $\omega(t,x)=\omega(t,X_t^{-1}(x))$. 
Then, using Theorem 4.2 in \cite{MV98} we conclude
\begin{equation}
\label{global besov eq1}
\norm{\omega_0\circ X_t^{-1}}{B^0_{\infty,1}} \leq C(1+\log (\norm{\nabla X_t^{-1}}{L^{\infty}}\norm{\nabla X_t}{L^{\infty}}))\norm{\omega_0}{B^0_{\infty,1}}.
\end{equation}
Using the ODE for the flow map, by the chain rule we have that
$$
\nabla X_t(x)=I_{2\times 2} + \int_0^t \nabla u(\tau,X_{\tau}(x))\nabla X_{\tau}(x)\,d\tau,
$$
which implies
\begin{equation}
\label{global besov eq2}
\norm{\nabla X_t}{L^{\infty}}\leq \exp\left(\int_0^t \norm{\nabla u(\tau)}{L^{\infty}}\,d\tau\right) \leq \exp\left(C\left(\norm{\omega_0}{L^p}t+\int_0^t\norm{\omega(\tau)}{B^0_{\infty,1}}\,d\tau\right)\right).
\end{equation}
Doing the same calculation for the inverse flow map, multiplying the result with \eqref{global besov eq2}, taking the logarithm, and adding one yields
\begin{equation}
\label{global besov eq3}
1+\log(\norm{\nabla X_t^{-1}}{L^{\infty}}\norm{\nabla X_t}{L^{\infty}}) \leq 1+C\left(\norm{\omega_0}{L^p}t+\int_0^t \norm{\omega(\tau)}{B^0_{\infty,1}}\,d\tau\right).
\end{equation}
Then substituting \eqref{global besov eq3} into \eqref{global besov eq1} and applying Grönwall's Lemma yields
\begin{equation}
\norm{\omega(t)}{B^0_{\infty,1}}\leq C\norm{\omega_0}{B^0_{\infty,1}}\left(1+Ct\norm{\omega_0}{L^p}\right)e^{Ct\norm{\omega_0}{B^0_{\infty,1}}}.
\end{equation}
Applying Theorem \ref{BKM-type inequality Besov Space} allows us to extend $u$ to a global in-time solution of \eqref{euler velocity}.
\end{proof}

\section{Well-Posedness in Generalized Triebel-Lizorkin Spaces}\label{Section Triebel}
Having established well-posedness for the Euler equations in Besov spaces of generalized smoothness, we now extend the result to the Triebel–Lizorkin scale. Unlike Besov spaces, which measure dyadic pieces in sequence spaces, Triebel–Lizorkin spaces combine them pointwise before integration, and it naturally includes Sobolev spaces. In particular, recall that $F^{s,\psi}_{p,2}(\R^d)=W^{s,p,\psi}(\R^d)$, but Besov spaces only coincide with Sobolev spaces in the case of $B^{s,\psi}_{2,2}(\R^d)=H^{s,\psi}(\R^d)$. We show that the arguments developed in the Besov case adapt to this context, yielding parallel well-posedness results.
\subsection{Short Time Existence}
\begin{theorem}
\label{triebel lizorkin apriori}
Suppose that $u$ is a smooth solution to \eqref{euler velocity} on $[0,T]\times\R^d$ with $u_0 \in F^{s,\psi}_{p,q}(\R^d)$ where $s=d/p+1$, $p \in  (1,\infty)$, and  $q \in [1,\infty)$. Then for any $t \leq T$
\begin{equation}
\norm{u(t)}{F^{s,\psi}_{p,q}}\leq\norm{u_0}{F^{s,\psi}_{p,q}}\exp\left(C\int_0^t \norm{\nabla u(\tau)}{L^{\infty}}\,d\tau\right).
\end{equation}
\end{theorem}
\begin{proof} We follow the proof of Theorem \ref{euler a-priori estimate besov} until \eqref{pre L^p apriori besov}. We integrate \eqref{pre L^p apriori besov} in time to obtain
$$
\dot{\Delta}_j u(t,X_t(x)) = \dot{\Delta}_j u_0(X_t(x)) + \int_0^t \left\{ [ u \cdot \nabla, \dot{\Delta}_j]u(\tau,X_{\tau}(x)) + \dot{\Delta}_j\nabla p(\tau, X_{\tau}(x))\right\}\,d\tau.
$$
Multiplying by $2^{js}\psi(2^j)$, taking the $l^q(\mathbb{Z})$ norm in $j$, and using Minkowski's inequality yields
\begin{multline}
\label{pre Lp apriori tl}
\left(\sum_{j \in \mathbb{Z}} (2^{js}\psi(2^j)|\dot{\Delta}_ju(t,X_t(x))|)^q\right)^{\frac{1}{q}} \leq \left(\sum_{j \in \mathbb{Z}} (2^{js}\psi(2^j)|\dot{\Delta}_ju_0(X_t(x))|)^q\right)^{\frac{1}{q}}\\
+\int_0^t \left(\sum_{j \in \mathbb{Z}} (2^{js}\psi(2^j)|[ u \cdot \nabla, \dot{\Delta}_j]u(\tau,X_{\tau}(x))|)^q\right)^{\frac{1}{q}}\,d\tau + \int_0^t \left(\sum_{j \in \mathbb{Z}} (2^{js}\psi(2^j)|\dot{\Delta}_j\nabla p(\tau, X_{\tau}(x))|)^q\right)^{\frac{1}{q}}d\tau.
\end{multline}
Taking the $L^p$ norm of \eqref{pre Lp apriori tl}, using that $X_t$ is measure preserving, and Proposition \ref{Triebel Lizorkin Commutator Estimate} then gives
\begin{equation}
\norm{u(t)}{\dot{F}^{s,\psi}_{p,q}} \lesssim \norm{u_0}{\dot{F}^{s,\psi}_{p,q}} +\int_0^t \left\{\norm{\nabla u(\tau)}{L^{\infty}}\norm{u(\tau)}{\dot{F}^{s,\psi}_{p,q}} + \norm{\nabla p(\tau)}{\dot{F}^{s,\psi}_{p,q}}\right\}\,d\tau.
\end{equation}
Using Proposition \ref{TL multiplier theorem} and \ref{Leibniz rule TL} we see 
$$
\norm{\nabla p}{\dot{F
}^{s,\psi}_{p,q}} \leq \sum_{i,j=1}^d
\norm{\nabla (-\Delta)^{-1}\partial_iu_j\partial_ju_i}{\dot{F}^{s,\psi}_{p,q}}\lesssim \sum_{i,j=1}^d
\norm{\partial_iu_j\partial_ju_i}{\dot{F}^{s-1,\psi}_{p,q}}\lesssim \norm{\nabla u}{L^{\infty}}\norm{u}{\dot{F}^{s,\psi}_{p,q}}.
$$
We then conclude
\begin{equation}
\label{homogeneuous TL apriori estimate}
\norm{u(t)}{\dot{F}^{s,\psi}_{p,q}} \leq \norm{u_0}{\dot{F}^{s,\psi}_{p,q}} +C \int_0^t \norm{\nabla u(\tau)}{L^{\infty}}\norm{u(\tau)}{\dot{F}^{s,\psi}_{p,q}}\,d\tau.
\end{equation}
Adding the estimate
$$
\norm{u(t)}{L^p} \leq \norm{u_0}{L^p} +C \int_0^t \norm{\nabla u(\tau)}{L^{\infty}}\norm{u(\tau)}{L^p}\,d\tau
,$$
to \eqref{homogeneuous TL apriori estimate} gives
\begin{equation}
\label{nonhomogeneuous TL apriori estimate}
\norm{u(t)}{F^{s,\psi}_{p,q}} \lesssim \norm{u_0}{F^{s,\psi}_{p,q}} + \int_0^t \norm{\nabla u(\tau)}{L^{\infty}}\norm{u(\tau)}{F^{s,\psi}_{p,q}}\,d\tau.
\end{equation}
An application of Grönwall's Lemma yields the desired result.
\end{proof}

When proving that the approximating sequence is Cauchy in $F^{s-1,\psi}_{p,q}(\R^d)$ we need a few technical lemmas which we state below. 
\begin{lemma}
\label{bernstein in TL} Suppose that $f \in \dot{F}^{s,\psi}_{p,q}(\R^d)$ with $p \in (1,\infty) \text{ and } q \in[1,\infty)$, then 
$$
\norm{\Delta_jf}{\dot{F}^{s-1,\psi}_{p,q}} \lesssim 2^{-j}\norm{f}{\dot{F}^{s,\psi}_{p,q}}.
$$
\end{lemma}
\begin{proof}
We write
\begin{align*}
\norm{\Delta_jf}{\dot{F}^{s-1,\psi}_{p,q}} &= \norm{\left(\sum_{k \in \mathbb{Z}} \left(2^{(s-1)k}\psi(2^k)|\dot{\Delta}_k\Delta_jf|\right)^q\right)^{1/q}}{L^p}\\
&\lesssim
\norm{\left(\sum_{|k-j|\leq 1} \left(2^{(s-1)k}\psi(2^k)\mathcal{M}(|\dot{\Delta}_kf|)\right)^q\right)^{1/q}}{L^p}\\
&\lesssim
\norm{\left(\sum_{|k-j|\leq 1} \left(2^{(s-1)k}\psi(2^k)|\dot{\Delta}_kf|\right)^q\right)^{1/q}}{L^p}
\end{align*}
where in the first inequality we used the support of the Littlewood-Paley operators, Lemma \ref{convolution bound by maximal function}, and that $\norm{\varphi_j}{L^1}$ is independent of $j$, and in the second inequality we used Lemma \ref{vector maximal inequality}. Expanding out the sum we have
\begin{align*}
&\norm{\left(\sum_{|k-j|\leq 1} \left(2^{(s-1)k}\psi(2^k)|\dot{\Delta}_kf|\right)^q\right)^{1/q}}{L^p} \\&= \norm{\left(2^{(s-1)(j-1)q}\psi^q(2^{j-1})|\dot{\Delta}_{j-1}f|^q + 2^{(s-1)jq}\psi^q(2^{j})|\dot{\Delta}_{j}f|^q + 2^{(s-1)(j+1)q}\psi^q(2^{j+1})|\dot{\Delta}_{j+1}f|^q \right)^{1/q}}{L^p}\\
&= 2^{-j}\norm{\left(2\cdot2^{s(j-1)q}\psi^q(2^{j-1})|\dot{\Delta}_{j-1}f|^q + 2^{sjq}\psi^q(2^{j})|\dot{\Delta}_{j}f|^q + \frac{1}{2}\cdot2^{s(j+1)q}\psi^q(2^{j+1})|\dot{\Delta}_{j+1}f|^q \right)^{1/q}}{L^p}\\
&\leq 2^{-j+1}\norm{\left(2^{s(j-1)q}\psi^q(2^{j-1})|\dot{\Delta}_{j-1}f|^q + 2^{sjq}\psi^q(2^{j})|\dot{\Delta}_{j}f|^q + 2^{s(j+1)q}\psi^q(2^{j+1})|\dot{\Delta}_{j+1}f|^q \right)^{1/q}}{L^p}\\
&\lesssim 2^{-j}\norm{f}{\dot{F}^{s,\psi}_{p,q}}.
\end{align*}
\end{proof}
\begin{lemma}
\label{dumb lemma}
Suppose that $s>1$, $p,q \in (1,\infty)$, and $\psi \in \mathcal{M}$ is non-decreasing. Then $F^{s,\psi}_{p,q}(\R^d) \hookrightarrow F^{1}_{p,2}(\R^d)$
\end{lemma}
\begin{proof}
We analyze two separate cases. If $q \leq 2$ we use a $l^q(\mathbb{Z}_{\geq -1})$ embedding and the inequality $2^{k} \leq C2^{ks}\psi(2^k)$ for $s \geq 1$ to see that
$$
\norm{\left(\sum_{k=-1}^{\infty}(2^k|\Delta_kf|^2\right)^{1/2}}{L^p} \leq C\norm{\left(\sum_{k=-1}^{\infty} \left(2^{ks}\psi(2^k)|\Delta_kf|\right)^q\right)^{1/q}}{L^p}.
$$
If $q>2$ then we use Hölder's inequality to observe that 
$$
\norm{\left(\sum_{k=-1}^{\infty}(2^k|\Delta_kf|^2\right)^{1/2}}{L^p} \leq \norm{\frac{2^{k(1-s)}}{\psi(2^k)}}{l^{\frac{2q}{q-2}}}\norm{\left(\sum_{k=-1}^{\infty} \left(2^{ks}\psi(2^k)|\Delta_kf|\right)^q\right)^{1/q}}{L^p}.
$$
In either case for $s>1$ we have shown that $\norm{f}{F^{1}_{p,2}} \lesssim \norm{f}{F^{s,\psi}_{p,q}}$.
\end{proof}
We are now ready to establish short-time existence in the space $F^{s,\psi}_{p,q}(\R^d)$.
\begin{theorem}\label{short time euler triebel lizorkin} Suppose $u_0 \in F^{s,\psi}_{p,q}(\R^d)$, $d\geq 2$, $s=d/p+1$, $p,q \in (1,\infty)$, and $\psi \in \mathcal{M}_{p'}$. Then there exists a $T=T(\norm{u_0}{F^{s,\psi}_{p,q}})>0$ such that a unique solution
$$
u \in C([0,T]; F^{s,\psi}_{p,q}(\R^d)),
$$
of the system \eqref{euler velocity} exists.
\end{theorem}
\begin{proof}
We follow the proof of Theorem \ref{short time euler besov} closely. We summarize the steps below.
\\
\\
\textbf{The Approximating Sequence:} We generate a sequence $(u^{(n)},p^{(n)})$ which satisfies \eqref{approximating sequence PDE} and \eqref{pressure approximation}.
\\
\\
\textbf{Uniform Bounds }We let $X^s_T= C([0,T];F^{s,\psi}_{p,q}(\R^d))$. Following the proof of Theorem \ref{triebel lizorkin apriori} we see the solution $(u^{(n)},p^{(n)})$ to \eqref{approximating sequence PDE}  satisfies
\begin{equation}
\label{TL uniform bound}
\norm{u^{(n)}(t)}{F^{s,\psi}_{p,q}} \leq C\left(\norm{u_0}{F^{s,\psi}_{p,q}} + \int_0^t \left\{ \norm{u^{(n-1)}(\tau)}{F^{s,\psi}_{p,q}}^2 + \norm{u^{(n-1)}(\tau)}{F^{s,\psi}_{p,q}}\norm{u^{(n)}(\tau)}{F^{s,\psi}_{p,q}}\right\}\,d\tau\right).
\end{equation}
Following the steps of Theorem \ref{short time euler besov} we see
\begin{equation}
\label{uniform bound TL}
\norm{u^{(n)}}{X_T^s}\leq 2\norm{u_0}{F^{s,\psi}_{p,q}}, \text{ for all } n \geq 0 \text{ and } T \leq T_0,
\end{equation}
where $T_0$ satisfies \eqref{Short time T}.
\\
\\
\noindent
\textbf{$u^{(n)}$ is Cauchy is $X^{s-1}_T$:} Without loss of generality, assume that $n \geq m$. We see that the difference of $u^{(n)}$ and $u^{(m)}$ must satisfy the following differential equation
\begin{equation}
\label{difference eq 1 TL}
\begin{cases}
\partial_t(u^{(n)}-u^{(m)})+(u^{(n-1)}\cdot\nabla)(u^{(n)}-u^{(m)}) = - ((u^{(n-1)}-u^{(m-1)})\cdot\nabla)u^{(m)} - \nabla (p^{(n-1)}-p^{(m-1)})\\
\left(u^{(n)}-u^{(m)}\right)\vert_{t=0}=\sum_{j=1}^{n-m}\Delta_{j+m}u_0.
\end{cases}
\end{equation}
In what follows, we let $X_{t,n-1}$ be the corresponding flow map for $u^{(n-1)}$. Applying $\dot{\Delta}_j$ to \eqref{difference eq 1 TL} and writing the result in Lagrangian coordinates yields
\begin{multline}
\label{difference eq 2 TL}
\frac{d}{dt}\left( \dot{\Delta}_j(u^{(n)}-u^{(m)})(t,X_{t,n-1}(x))\right) = -\dot{\Delta}_j\left(((u^{(n-1)}-u^{(m-1)})\cdot\nabla)u^{(m)}\right)(t,X_{t,n-1}(x))\\ - \dot{\Delta}_j\nabla(p^{(n-1)}-p^{(m-1)})(t,X_{t,n-1}(x)) + \left(\left([u^{(n-1)}\cdot\nabla,\dot{\Delta}_j](u^{(n)}-u^{(m)})\right)(t,X_{t,n-1}(x))\right).
\end{multline}
We then integrate \eqref{difference eq 2 TL} in time, multiply by $2^{j(s-1)}\psi(2^j)$, take the $l^q(\mathbb{Z})$ norm, apply the triangle inequality, and use Minkowski's inequality to see
\begin{align*}
&\left(\sum_{j \in \mathbb{Z}}\left(2^{j(s-1)}\psi(2^j)\left|\dot{\Delta}_j(u^{(n)}-u^{(m)})(t,X_{t,n-1}(x)) \right|\right)^q\right)^{\frac{1}{q}}
\\ &\leq \left(\sum_{j \in \mathbb{Z}}\left(2^{j(s-1)}\psi(2^j)\left|\dot{\Delta}_j(u^{(n)}-u^{(m)})(0) \right|\right)^q\right)^{\frac{1}{q}}
\\ &\quad+\int_0^t \left(\sum_{j \in \mathbb{Z}}\left(2^{j(s-1)}\psi(2^j)\left|\dot{\Delta}_j\left(((u^{(n-1)}-u^{(m-1)})\cdot\nabla)u^{(m)}\right|\right)(\tau,X_{\tau,n-1}(x))\right)^q\,d\tau\right)^{\frac{1}{q}}
\\ &\quad+\int_0^t \left(\sum_{j \in \mathbb{Z}}\left(2^{j(s-1)}\psi(2^j)\left|\left(\left([u^{(n-1)}\cdot\nabla,\dot{\Delta}_j](u^{(n)}-u^{(m)})\right)(\tau, X_{\tau,n-1}(x))\right)\right|\right)^q\,d\tau\right)^{\frac{1}{q}}\\&\quad+\int_0^t \left(\sum_{j \in \mathbb{Z}}\left(2^{j(s-1)}\psi(2^j)|\dot{\Delta}_j\nabla(p^{(n-1)}-p^{(m-1)})(\tau,X_{\tau,n-1}(x))|\right)^q\,d\tau\right)^{\frac{1}{q}}.
\end{align*}
We then take the $L^p(\R^d)$ norm and use that $X_{\tau,n-1}(x)$ is measure preserving to see
\begin{equation}
\begin{split}
\label{I,II,III,IV estimate TL short time}
\|(u^{(n)}-&u^{(m)})(t)\|_{\dot{F}^{s-1,\psi}_{p,q}} \leq \norm {(u^{(n)}-u^{(m)})(0)}{\dot{F}^{s-1,\psi}_{p,q}}\\ 
&+\int_0^t \norm{\norm{2^{j(s-1)}\psi(2^j)\left|\left(\left([u^{(n-1)}\cdot\nabla,\dot{\Delta}_j](u^{(n)}-u^{(m)})\right)(\tau,X_{\tau}^{n-1}(x))\right)\right|}{l^q(\mathbb{Z})}}{L^p}\,d\tau.
\\ &+ \int_0^t \norm{\left(((u^{(n-1)}-u^{(m-1)})\cdot\nabla)u^{(m)}\right)(\tau)}{\dot{F}^{s-1,\psi}_{p,q}}\,d\tau+\int_0^t \norm{\nabla(p^{(n-1)}-p^{(m-1)})(\tau)}{\dot{F}^{s-1,\psi}_{p,q}}\,d\tau
\\ &=I+II+III+IV.
\end{split}
\end{equation}
We now estimate the four terms separately.
\\
\\
\textbf{Estimate of $I$:} Utilizing Lemma \ref{bernstein in TL} we have 
\begin{equation}
\label{I estimate TL}
    I=\norm{\sum_{j=1}^{n-m} \Delta_{j+m}u_0}{\dot{F}^{s-1,\psi}_{p,q}}
    \leq \sum_{j=1}^{n-m} \norm{\Delta_{j+m}u_0}{\dot{F}^{s-1,\psi}_{p,q}}
    \lesssim \sum_{j=1}^{n-m} 2^{-(j+m)}\norm{u_0}{\dot{F}^{s,\psi}_{p,q}}
    \lesssim 2^{-m}\norm{u_0}{\dot{F}^{s,\psi}_{p,q}}.
\end{equation}
\\
\textbf{Estimate of $II$:}
Using Proposition \ref{Triebel Lizorkin Commutator Estimate} we have 
\begin{equation}
\label{II estimate i TL}
II \lesssim \int_0^t \left(\norm{\nabla u^{(n-1)}(\tau)}{L^{\infty}}\norm{(u^{(n)}-u^{(m)})(\tau)}{F^{s-1,\psi}_{p,q}} + \norm{(u^{(n)}-u^{(m)})(\tau)}{L^{\infty}}\norm{\nabla u^{(n-1)}(\tau)}{F^{s-1,\psi}_{p,q}}\right)\,d\tau.
\end{equation}
We note that
$$\norm{\nabla u^{(n-1)}(\tau)}{L^{\infty}} \lesssim \norm{u^{(n-1)}(\tau)}{F^{s,\psi}_{p,q}}\quad \text{ and }\quad\norm{(u^{(n)}-u^{(m)})(\tau)}{L^{\infty}} \lesssim \norm{(u^{(n-1)} - u^{(m-1)})(\tau)}{F^{s-1,\psi}_{p,q}}.
$$
Then since $W^{1,p}(\R^d)=F^1_{p,2}(\R^d)$, we utilize Lemma \ref{dumb lemma} to see that
\begin{align*}
\norm{\nabla u^{(n-1)}(\tau)}{F^{s-1\psi}_{p,q}}&\sim \norm{\nabla u^{(n-1)}(\tau)}{L^p}+\norm{\nabla u^{(n-1)}(\tau)}{\dot{F}^{s-1,\psi}_{p,q}}\\ &\lesssim \norm{u^{(n-1)}(\tau)}{F^1_{p,2}} + \norm{\nabla u^{(n-1)}(\tau)}{\dot{F}^{s-1,\psi}_{p,q}}\\ &\lesssim \norm{u^{(n-1)}(\tau)}{F^{s,\psi}_{p,q}}.\end{align*} We then bound $II$ as follows
\begin{equation}
\label{II estiamte TL}
II \lesssim \int _0^t \norm{u^{(n-1)}(\tau)}{F^{s,\psi}_{p,q}}\norm{(u^{(n)}-u^{(m)})(\tau)}{F^{s-1,\psi}_{p,q}}\,d\tau.
\end{equation}
\textbf{Estimate of $III$:} Using Proposition \ref{Leibniz rule TL} we bound $III$ by

\begin{multline}
    \label{III estimate TL}
III=\int_0^t \norm{\left(((u^{(n-1)}-u^{(m-1)})\cdot\nabla)u^{(m)}\right)(\tau)}{\dot{F}^{s-1,\psi}_{p,q}}\,d\tau \\
\lesssim \int_0^t \left(\norm{(u^{(n-1)}-u^{(m-1)})(\tau)}{\dot{F}^{s-1,\psi}_{p,q}} \norm{\nabla u^{(m)}(\tau)}{L^{\infty}} + \norm{\nabla u^{(m)}(\tau)}{\dot{F}^{s-1,\psi}_{p,q}}\norm{(u^{(n-1)}-u^{(m-1)})(\tau)}{L^{\infty}}\right)\,d\tau\\
\lesssim \int_0^t \norm{(u^{(n-1)}-u^{(m-1)})(\tau)}{\dot{F}^{s-1,\psi}_{p,q}}\norm{u^{(m)}(\tau)}{F^{s,\psi}_{p,q}}\,d\tau.
\end{multline}
\textbf{Estimate of $IV$:} Finally, we handle the pressure  term. We first recall the identity
\begin{align*}
\nabla p^{(n-1)}-\nabla p^{(m-1)} &= \sum_{i,j=1}^d (-\Delta)^{-1}\nabla \partial_j\left\{\left(\partial_iu_j^{(n-1)}\right)\left(u_i^{(n-1)}-u_i^{(m-1)}\right)\right\}\\ 
&\hspace{1cm}+ \sum_{i,j=1}^d (-\Delta)^{-1}\nabla \partial_i \left\{\left(\partial_ju_i^{(m-1)}\right)\left(u_j^{(n-1)}-u_j^{(m-1)}\right)\right\}.
\end{align*}
Then, using Proposition \ref{TL multiplier theorem} we find 
\begin{multline}
    \label{IV estimate TL}
    IV \lesssim \int_0^t \left(\norm{(\nabla u^{(n-1)} \cdot (u^{(n-1)}-u^{(m-1)}))(\tau)}{\dot{F}^{s-1,\psi}_{p,q}}+\norm{(\nabla u^{(m-1)} \cdot (u^{(n-1)}-u^{(m-1)}))(\tau)}{\dot{F}^{s-1,\psi}_{p,q}}  \right)\,d\tau\\
    \lesssim \int_0^t \left( \norm{u^{(n-1)}(\tau)}{F^{s,\psi}_{p,q}} + \norm{u^{(m-1)}(\tau)}{F^{s,\psi}_{p,q}}\right)\norm{(u^{(n-1)}-u^{(m-1)})(\tau)}{F^{s-1,\psi}_{p,q}}\,d\tau.
\end{multline}
Using the estimates  \eqref{I estimate TL}, \eqref{II estiamte TL}, \eqref{III estimate TL}, and \eqref{IV estimate TL} we see that 
\begin{multline}
\label{TL Cauchy Estimate}
\norm{(u^{(n)}-u^{(m)})(t)}{\dot{F}^{s-1,\psi}_{p,q}} \lesssim 2^{-m}\norm{u_0}{\dot{F}^{s,\psi}_{p,q}}+\int_0^t \norm{u^{(n-1)}(\tau)}{F^{s,\psi}_{p,q}}\norm{(u^{(n)}-u^{(m)})(\tau)}{F^{s-1,\psi}_{p,q}}\,d\tau\\ 
+\int_0^t \left(\norm{u^{(m)}(\tau)}{F^{s,\psi}_{p,q}}+\norm{u^{(n-1)}(\tau)}{F^{s,\psi}_{p,q}}+\norm{u^{(m-1)}(\tau)}{F^{s,\psi}_{p,q}}\right)\norm{(u^{(n-1)}-u^{(m-1)})(\tau)}{F^{s-1,\psi}_{p,q}}\,d\tau.
\end{multline}
The proof then finishes in the exact same manner as Theorem \ref{short time euler besov}.
\end{proof}
\subsection{Global Existence} 
\begin{theorem}
\label{bkm triebel lizorkin}
Let $p,q\in (1,\infty)$, $s=d/p+1$, and $\psi \in \mathcal{M}_{p'}$. Then a solution $u \in C([0,T];F^{s,\psi}_{p,q}(\R^d))$ blows up at time $T^*>T$, namely
$$
\limsup_{t \to T^{*}} \norm{u(t)}{F^{s,\psi}_{p,q}}=\infty
$$
if and only if 
$$
\int_0^{T^*} \norm{\omega (t)}{\dot{B}^0_{\infty,1}}\,dt=\infty.
$$
\end{theorem}
\begin{proof}
From Theorem \ref{triebel lizorkin apriori} we have
\begin{equation}
\label{bkm tl 1}
\norm{u(t)}{F^{s,\psi}_{p,q}} \leq \norm{u_0}{F^{s,\psi}_{p,q}}\exp\left(C\int_0^t \norm{\nabla u(\tau)}{L^{\infty}}\,d\tau\right).
\end{equation}
Then, using the estimate 
$$
\norm{\nabla u(t)}{L^\infty} \leq \norm{\omega_0}{L^p} + \norm{\omega(t)}{\dot{B}^0_{\infty,1}}
$$
yields
\begin{equation}
\label{bkm tl 2}
\norm{u(t)}{F^{s,\psi}_{p,q}}\leq \norm{u_0}{F^{s,\psi}_{p,q}}e^{Ct\norm{\omega_0}{L^p}}\exp\left(C\int_0^t \norm{\omega(\tau)}{\dot{B}^0_{\infty,1}}\,d\tau\right)
\end{equation}
which shows the ``only if" part of the theorem. To see the ``if" part, thanks to Proposition \ref{key triebel lizorkin embedding} we have the following series of inequalities
$$
\int_0^T \norm{\omega(t)}{\dot{B}^0_{\infty,1}}\,dt \leq T\sup_{0 \leq t \leq T} \norm{\omega(t)}{\dot{B}^0_{\infty,1}}\leq CT\sup_{0\leq t \leq T} \norm{u(t)}{F^{s,\psi}_{p,q}}.
$$
\end{proof}

We skip the proof of the following corollary which is identical to that of Corollary \ref{global well posedness in Besov Spaces}.
\begin{corollary}
\label{global well posedness in TL Spaces}
Let $p,q \in (1,\infty), \psi \in \mathcal{M}_{p'} $, and $s=\frac{2}{p}+1$. Given $u_0 \in F^{s,\psi}_{p,q}(\R^2)$ there is a unique solution to \eqref{euler velocity} with $u \in C([0,\infty);F^{s,\psi}_{p,q}(\R^2))$.
\end{corollary}

\appendix 
\section{\\ Calculus Inequalities}
When establishing calculus inequalities in the spaces $B^{s,\psi}_{p,q}(\R^d)$ and $F^{s,\psi}_{p,q}(\R^d)$, we will lean heavily on the fact that $\psi$ is increasing, continuous, and slowly varying. Recall that $\psi$ is slowly varying if 
$$
\lim_{t \to \infty} \frac{\psi(\lambda t)}{\psi(t)}=1 \text{ for all $\lambda >0$.}
$$
In particular, since $\psi$ is continuous, this means that for each $\lambda>0$ there exists some constant $C_{\lambda,\psi}>0$ such that 
$
|\psi(\lambda t)| \leq C_{\lambda,\psi}|\psi(t)|.
$
In particular, if $j-j'\leq l$, using that $\psi$ is increasing and the previous remark yields
\begin{equation}
\label{j j' estimate}
\psi(2^j) \leq \psi(2^{j'+l}) \leq C_{l,\psi} \psi(2^{j'}).
\end{equation}
\subsection{Leibniz Rule}
\begin{lemma}
\label{Leibniz Rule in Besov Spaces}
Let $s>0$, $p,q \in [1,\infty]$ and $\frac{1}{p_1}+\frac{1}{p_2}=\frac{1}{r_1}+\frac{1}{r_2}=\frac{1}{p}$. Then,
$$
\norm{fg}{\dot{B}^{s,\psi}_{p,q}} \lesssim \norm{f}{L^{p_1}}\norm{g}{\dot{B}^{s,\psi}_{p_2,q}} + \norm{g}{L^{r_1}}\norm{f}{\dot{B}^{s,\psi}_{r_2,q}},
$$
and 
$$
\norm{fg}{{B}^{s,\psi}_{p,q}} \lesssim \norm{f}{L^{p_1}}\norm{g}{{B}^{s,\psi}_{p_2,q}} + \norm{g}{L^{r_1}}\norm{f}{{B}^{s,\psi}_{r_2,q}}.
$$
Consequently, if $B^{s,\psi}_{p,q}(\mathbb{R}^d)\hookrightarrow L^\infty(\mathbb{R}^d)$, then $$\norm{fg}{B^{s,\psi}_{p,q}}\lesssim\norm{f}{B^{s,\psi}_{p,q}}\norm{g}{B^{s,\psi}_{p,q}}.$$
\end{lemma}
\begin{proof}
We first consider the homogeneous spaces $\dot{B}^{s,\psi}_{p,q}(\R^d)$. We use Bony's paraproduct decomposition to write $fg=\dot{T}_fg+\dot{T}_gf+\dot{R}(f,g)$. Then 
\begin{align*}
\dot{\Delta}_j(fg) &= \dot{\Delta}_j(\dot{T}_fg+\dot{T}_gf+\dot{R}(f,g))\\
&= \sum_{j'\in \mathbb{Z}}\dot{\Delta}_j (\dot{S}_{j'-2}f\dot{\Delta}_{j'}g) + \sum_{j'\in \mathbb{Z}} \dot{\Delta}_j (\dot{S}_{j'-2}g\dot{\Delta}_{j'}f)
 + \sum_{|i'-j'|\leq 1} \dot{\Delta}_j (\dot{\Delta}_{i'}f\dot{\Delta}_{j'}g).\\
&\eqdef I + II + III.
\end{align*}
Due to the support of the Littlewood-Paley operators we see that 
\begin{equation}
\label{besov leibniz est 1} 
\norm{I}{L^p} \lesssim \sum_{|j-j'| \leq 4} \norm{\dot{S}_{j'-2}f}{L^{p_1}}\norm{\dot{\Delta}_jg}{L^{p_2}} \lesssim \norm{f}{L^{p_1}}\sum_{|j-j'| \leq 4}\norm{\dot{\Delta}_jg}{L^{p_2}}.
\end{equation}
Similarly, we see that 

\begin{equation}
\label{besov leibniz est 2} 
\norm{II}{L^p} \lesssim \sum_{|j-j'| \leq 4} \norm{\dot{S}_{j'-2}g}{L^{r_1}}\norm{\dot{\Delta}_jf}{L^{r_2}} \lesssim \norm{g}{L^{r_1}}\sum_{|j-j'| \leq 4}\norm{\dot{\Delta}_jf}{L^{r_2}},
\end{equation}
and 
\begin{equation}
\label{besov leibniz est 3}
\norm{III}{L^p} \lesssim \sum_{j' \geq j -5}\norm{f}{L^{p_1}}\norm{\dot{\Delta}_{j'}g}{L^{p_2}}.
\end{equation}
Then, by Minkowski's inequality we see that 
\begin{align*}
\label{besov leibniz est 4}
\norm{fg}{\dot{B}^{s,\psi}_{p,q}}&= \left(\sum_{j \in \mathbb{Z}}\left( 2^{js}\psi(2^j) \norm{\dot{\Delta}_j (fg)}{L^p}\right)^q\right)^{\frac{1}{q}}\lesssim \norm{f}{L^{p_1}}\left(\sum_{j \in \mathbb{Z}}\left(\sum_{|j-j'|\leq 4} 2^{js}\psi(2^j) \norm{\dot{\Delta}_{j'}g}{L^{p_2}}\right)^q\right)^{\frac{1}{q}} \\
& \hspace{-1.5cm}+\norm{g}{L^{r_1}}\left(\sum_{j \in \mathbb{Z}}\left(\sum_{|j-j'| \leq 4} 2^{js}\psi(2^j) \norm{\dot{\Delta}_{j'}f}{L^{r_2}}\right)^q\right)^{\frac{1}{q}} + \norm{f}{L^{p_1}}\left(\sum_{j \in \mathbb{Z}} \left(\sum_{j' \geq j-5} 2^{js}\psi(2^j) \norm{\dot{\Delta}_{j'}g}{L^{p_2}}\right)^q\right)^{\frac{1}{q}}\\ &\eqdef \{I\} + \{II\} + \{III\}.
\end{align*}
By Young's convolution inequality for series we estimate
\begin{align*}
\{I\} &= \norm{f}{L^{p_1}}\left(\sum_{j \in \mathbb{Z}}\left(\sum_{|j-j'|\leq 4} 2^{js}\psi(2^j)\norm{\dot{\Delta}_{j'}g}{L^{p_2}}\right)^q\right)^{\frac{1}{q}}\\
&\lesssim \norm{f}{L^{p_1}}\left(\sum_{j \in \mathbb{Z}}\left(\sum_{j' \in \mathbb{Z}} \chi_{|j-j'|\leq 4}2^{(j-j')s}\psi(2^{j'}) 2^{j's}\norm{\dot{\Delta}_{j'}g}{L^{p_2}}\right)^q\right)^{\frac{1}{q}}\\
&\leq \left(\sum_{j \in \mathbb{Z}}\chi_{|j|\leq 4} 2^{js}\right)\norm{f}{L^{p_1}}\norm{g}{\dot{B}^{s,\psi}_{p_2,q}}\\
&\lesssim \norm{f}{L^{p_1}}\norm{g}{\dot{B}^{s,\psi}_{p_2,q}}
\end{align*}
An identical argument shows
$$
\{II\} \lesssim \norm{g}{L^{r_1}}\norm{f}{\dot{B}^{s,\psi}_{r_2,q}}.
$$
It remains to estimate the third term. Using the fact $\psi$ is increasing with Young's convolution inequality we have
\begin{align*}
\{III\}&=\norm{f}{L^{p_1}}\left(\sum_{j \in \mathbb{Z}} \left(\sum_{j' \geq j-5} 2^{js}\psi(2^j) \norm{\dot{\Delta}_{j'}g}{L^{p_2}}\right)^q\right)^{\frac{1}{q}}\\
&\leq \norm{f}{L^{p_1}}\left(\sum_{j \in \mathbb{Z}} \left(\sum_{j \in \mathbb{Z}} \chi_{\{j-j'\leq 5\}}2^{(j-j')s}\psi(2^{j'})2^{j's} \norm{\dot{\Delta}_{j'}g}{L^{p_2}}\right)^q\right)^{\frac{1}{q}}\\
&\leq
\left(\sum_{j \in \mathbb{Z}} \chi_{\{j \leq 5\}}2^{js}\right) \norm{f}{L^{p_1}}\norm{g}{\dot{B}^{s,\psi}_{p_2,q}}\\
&\lesssim \norm{f}{L^{p_1}}\norm{g}{\dot{B}^{s,\psi}_{p_2,q}}.
\end{align*}
We have proven
$$
\norm{fg}{\dot{B}^{s,\psi}_{p,q}} \lesssim \norm{f}{L^{p_1}}\norm{g}{\dot{B}^{s,\psi}_{p_2,q}}+ \norm{g}{L^{r_1}}\norm{f}{\dot{B}^{s,\psi}_{r_2,q}}.
$$
The nonhomogeneous inequality is obtained by adding the following estimate
$$
\norm{fg}{L^p} \leq \frac{1}{2}\left(\norm{f}{L^{p_1}}\norm{g}{L^{p_2}} + \norm{g}{L^{r_1}}\norm{f}{L^{r_2}}\right).
$$
\end{proof}
We next establish a Leibniz type rule in the spaces of Triebel-Lizorkin type.
\begin{lemma}
\label{Leibniz rule TL}
Suppose that $s>0$ and  $p,q \in [1,\infty]$. Then the following inequalities hold:
$$
\norm{fg}{\dot{F}^{s,\psi}_{p,q}} \lesssim \norm{f}{L^{\infty}}\norm{g}{\dot{F}^{s,\psi}_{p,q}} + \norm{g}{L^{\infty}}\norm{f}{\dot{F}^{s,\psi}_{p,q}},
$$
and
$$
\norm{fg}{F^{s,\psi}_{p,q}} \lesssim \norm{f}{L^{\infty}}\norm{g}{F^{s,\psi}_{p,q}} + \norm{g}{L^{\infty}}\norm{f}{F^{s,\psi}_{p,q}}.
$$
\end{lemma}
\begin{proof}
As before, we first consider the homogeneous spaces $\dot{F}^{s,\psi}_{p,q}(\R^d)$. Again, we use Bony's paraproduct decomposition to write $fg=\dot{T}_fg+\dot{T}_gf+\dot{R}(f,g)$. Then
\begin{align*}
\dot{\Delta}_j(fg) &= \dot{\Delta}_j(\dot{T}_fg+\dot{T}_gf+\dot{R}(f,g))\\
&= \sum_{j'\in \mathbb{Z}}\dot{\Delta}_j (\dot{S}_{j'-2}f\dot{\Delta}_{j'}g) + \sum_{j' \in \mathbb{Z}} \dot{\Delta}_j (\dot{S}_{j'-2}g\dot{\Delta}_{j'}f)
 + \sum_{|i'-j'|\leq 1} \dot{\Delta}_j (\dot{\Delta}_{i'}f\dot{\Delta}_{j'}g).\\
&\eqdef I + II + III.
\end{align*}
Due to the support of the $\dot{\Delta}_j$ operators and Minkowski's inequality  we see that 
\begin{align*}
&\norm{fg}{\dot{F}^{s,\psi}_{p,q}} \lesssim \norm{\left(\sum_{j \in \mathbb{Z}}\left(\sum_{|j-j'| \leq 4} 2^{js}\psi(2^j)|\dot{\Delta}_j(\dot{S}_{j'-2}f\dot{\Delta}_{j'}g)|\right)^q\right)^\frac{1}{q}}{L^p}\\
&+
\norm{\left(\sum_{j \in \mathbb{Z}}\left(\sum_{|j-j'| \leq 4} 2^{js}\psi(2^j)|\dot{\Delta}_j(\dot{S}_{j'-2}g\dot{\Delta}_{j'}f)|\right)^q\right)^\frac{1}{q}}{L^p}\\ &+ \norm{\left(\sum_{j \in \mathbb{Z}}\left( \sum_{\substack {|i'-j'| \leq 1 \\ \max\{i',j'\} \geq j-5}} 2^{js}\psi(2^j) |\dot{\Delta}_j(\dot{\Delta}_{i'}f \dot{\Delta}_{j'}g )|\right)^q \right)^{\frac{1}{q}}}{L^p}\\
&=\{I\}+\{II\}+\{III\}.
\end{align*}
We estimate $\{I\}$ as follows:
\begin{align*}
\{I\}&=\norm{\left(\sum_{j \in \mathbb{Z}}\left(\sum_{|j-j'| \leq 4} 2^{js}\psi(2^j)|\dot{\Delta}_j(\dot{S}_{j'-2}f\dot{\Delta}_{j'}g)|\right)^q\right)^\frac{1}{q}}{L^p}\\
&= \norm{\left(\sum_{j \in \mathbb{Z}}\left(\sum_{|j-j'| \leq 4} 2^{js}\psi(2^j)\left|2^{jd}\int_{\R^d} \varphi(2^{j}(\cdot-y))(\dot{S}_{j'-2}f\dot{\Delta}_{j'}g)(y)dy \right|\right)^q\right)^\frac{1}{q}}{L^p}\\
&= \norm{\left(\sum_{j \in \mathbb{Z}}\left(\sum_{|j-j'| \leq 4} 2^{js}\psi(2^j)\left|\int_{\R^d} \varphi(z)(\dot{S}_{j'-2}f\dot{\Delta}_{j'}g)(\cdot-2^{-j}z)dz \right|\right)^q\right)^\frac{1}{q}}{L^p}\\
&\leq
\norm{\left( \sum_{j \in \mathbb{Z}}\left( \int_{\R^d} \sum_{j' \in \mathbb{Z}} \chi_{\{|j-j'| \leq 4\}} 2^{js}\psi(2^j) \left|\varphi(z) (\dot{S}_{j'-2}f\dot{\Delta}_jg)(\cdot-2^{-j}z)\right|dz \right)^q \right)^{\frac{1}{q}}}{L^p}\\
&\leq
\norm{\int_{\R^d} \left( \sum_{j \in \mathbb{Z}} \left(\sum_{j' \in \mathbb{Z}} \chi_{\{|j-j'| \leq 4\}} 2^{js}\psi(2^j) \left|\varphi(z) (\dot{S}_{j'-2}f\dot{\Delta}_jg)(\cdot-2^{-j}z)\right| \right)^q\right)^{\frac{1}{q}}dz}{L^p}\\
&\lesssim
\norm{f}{L^{\infty}}\norm{\int_{\R^d}|\varphi(z)| \left( \sum_{j \in \mathbb{Z}} \left(\sum_{j' \in \mathbb{Z}} \chi_{\{|j-j'| \leq 4\}} 2^{js}\psi(2^j) \left|(\dot{\Delta}_jg)(\cdot-2^{-j}z)\right| \right)^q\right)^{\frac{1}{q}}dz}{L^p}\\
&
\leq
\norm{f}{L^{\infty}}\int_{\R^d}|\varphi(z)| \norm{\left( \sum_{j \in \mathbb{Z}} \left(\sum_{j' \in \mathbb{Z}} \chi_{\{|j-j'| \leq 4\}} 2^{js}\psi(2^j) \left|(\dot{\Delta}_jg)(\cdot-2^{-j}z)\right| \right)^q\right)^{\frac{1}{q}}}{L^p}dz\\
&\lesssim 
\norm{f}{L^{\infty}}\norm{\left( \sum_{j \in \mathbb{Z}} \left(\sum_{j' \in \mathbb{Z}} \chi_{\{|j-j'| \leq 4\}} 2^{js}\psi(2^j) \left|\dot{\Delta}_jg(\cdot)\right| \right)^q\right)^{\frac{1}{q}}}{L^p},
\end{align*}
where in the last inequality we used the translation invariance of the $L^p$ norm. Now, observe that by Young's convolution inequality for series 
\begin{align*}
&\norm{f}{L^{\infty}}\norm{\left( \sum_{j \in \mathbb{Z}} \left(\sum_{j' \in \mathbb{Z}} \chi_{\{|j-j'| \leq 4\}} 2^{js}\psi(2^j) \left|\dot{\Delta}_jg(\cdot)\right| \right)^q\right)^{\frac{1}{q}}}{L^p} \\&\lesssim \norm{f}{L^{\infty}}\norm{\left( \sum_{j \in \mathbb{Z}} \left(\sum_{j' \in \mathbb{Z}} \chi_{\{|j-j'| \leq 4\}} 2^{(j-j')s}2^{j's}\psi(2^{j'}) \left|\dot{\Delta}_jg(\cdot)\right| \right)^q\right)^{\frac{1}{q}}}{L^p}\\
&\leq \left(\sum_{|j| \leq 4} 2^{js}\right)\norm{f}{L^{\infty}} \norm{g}{\dot{F}^{s,\psi}_{p,q}}.
\end{align*}
Thus, $\{I\} \lesssim \norm{f}{L^{\infty}}\norm{g}{\dot{F}^{s,\psi}_{p,q}}$. An identical argument yields
$$
\{II\} \lesssim \norm{g}{L^{\infty}}\norm{f}{\dot{F}^{s,\psi}_{p,q}}.
$$
It remains to estimate $\{III\}$. We have that 
\begin{align*}
\{III\} &=  \norm{\left(\sum_{j \in \mathbb{Z}}\left( \sum_{\substack {|i'-j'| \leq 1 \\ \max\{i',j'\} \geq j-5}} 2^{js}\psi(2^j) |\dot{\Delta}_j(\dot{\Delta}_{i'}f \dot{\Delta}_{j'}g) |\right)^q \right)^{\frac{1}{q}}}{L^p}\\
&= \norm{\left(\sum_{j \in \mathbb{Z}}\left( \sum_{\substack {|i'-j'| \leq 1 \\ \max\{i',j'\} \geq j-5}} 2^{js}\psi(2^j) \left|2^{jd}\int_{\R^d}\varphi(2^{j}(\cdot-y))(\dot{\Delta}_{i'}f \dot{\Delta}_{j'}g)(y) dy \right|\right)^q \right)^{\frac{1}{q}}}{L^p}\\
&=\norm{\left(\sum_{j \in \mathbb{Z}}\left( \sum_{\substack {|i'-j'| \leq 1 \\ \max\{i',j'\} \geq j-5}} 2^{js}\psi(2^j) \left|\int_{\R^d}\varphi(z)(\dot{\Delta}_{i'}f \dot{\Delta}_{j'}g)(\cdot-2^{-j}z) dz \right|\right)^q \right)^{\frac{1}{q}}}{L^p}\\
&\lesssim 
\norm{f}{L^{\infty}}\norm{\left(\sum_{j \in \mathbb{Z}}\left( \int_{\R^d} \left|\sum_{j'\in \mathbb{Z}} \chi_{\{j'-j \geq 5\}}2^{js}\psi(2^j) \varphi(z)(\dot{\Delta}_{j'}g)(\cdot-2^{-j}z) \right|dz\right)^q \right)^{\frac{1}{q}}}{L^p}\\
&\leq 
\norm{f}{L^{\infty}} \norm{\int_{\R^d} |\varphi(z)|\left( \sum_{j \in \mathbb{Z}} \left( \sum_{j' \in \mathbb{Z}} \chi_{\{j'-j \geq 5\}}2^{js}\psi(2^j) |(\dot{\Delta}_{j'}g)(\cdot-2^{-j}z)|\right)^q\right)^{\frac{1}{q}}dz}{L^p}\\
&\leq
\norm{f}{L^{\infty}} \int_{\R^d} |\varphi(z)|\norm{\left( \sum_{j \in \mathbb{Z}} \left( \sum_{j' \in \mathbb{Z}} \chi_{\{j'-j \geq 5\}}2^{js}\psi(2^j) |(\dot{\Delta}_{j'}g)(\cdot-2^{-j}z)|\right)^q\right)^{\frac{1}{q}}}{L^p}dz\\
&\lesssim
\norm{f}{L^{\infty}}\norm{\left( \sum_{j \in \mathbb{Z}} \left( \sum_{j' \in \mathbb{Z}} \chi_{\{j'-j \geq 5\}}2^{js}\psi(2^j)|(\dot{\Delta}_{j'}g)(\cdot)|\right)^q\right)^{\frac{1}{q}}}{L^p},\\
\end{align*}
where in the last inequality we used the translation invariance of $L^p$. Applying Young's convolution inequality for series once again yields 
\begin{align*}
&\norm{f}{L^{\infty}}\norm{\left( \sum_{j \in \mathbb{Z}} \left( \sum_{j' \in \mathbb{Z}} \chi_{\{j'-j \geq 5\}}2^{js}\psi(2^j) |\dot{\Delta}_{j'}g(\cdot)|\right)^q\right)^{\frac{1}{q}}}{L^p}\\ &\lesssim \norm{f}{L^{\infty}}\norm{\left( \sum_{j \in \mathbb{Z}} \left( \sum_{j' \in \mathbb{Z}} \chi_{\{j'-j \geq 5\}}2^{(j-j')s} 2^{j's}\psi(2^{j'})|\dot{\Delta}_{j'}g(\cdot)|\right)^q\right)^{\frac{1}{q}}}{L^p}\\
&\lesssim \norm{f}{L^{\infty}}\norm{g}{\dot{F}^{s,\psi}_{p,q}}.
\end{align*}
We have shown 
$$
\norm{fg}{\dot{F}^{s,\psi}_{p,q}} \lesssim \norm{f}{L^{\infty}}\norm{g}{\dot{F}^{s,\psi}_{p,q}}+ \norm{g}{L^{\infty}}\norm{f}{\dot{F}^{s,\psi}_{p,q}}. 
$$
To obtain the nonhomogeneous estimate we add the following inequality
$$
\norm{fg}{L^p} \leq \frac{1}{2}\left(\norm{f}{L^{\infty}}\norm{g}{L^p} + \norm{g}{L^{\infty}}\norm{f}{L^p}\right). 
$$
\end{proof}
\subsection{Commutator Estimates}

\begin{proposition}
\label{Triebel Lizorkin Commutator Estimate}
Let $(p,q) \in (1,\infty)\times(1,\infty]$, or $p=q=\infty$. Let $u$ be a divergence free vector field and $\omega = \text{curl}(u)$. Then for $s>0$
\begin{equation}
\label{Bad TLCE}
\norm{\norm{2^{js}\psi(2^j)([u\cdot\nabla,\dot{\Delta}_j]\omega)}{l^q(\mathbb{Z})}}{L^p} \lesssim \norm{\nabla u}{L^{\infty}}\norm{\omega}{F^{s,\psi}_{p,q}}+\norm{\nabla \omega}{L^{\infty}}\norm{u}{F^{s,\psi}_{p,q}}.
\end{equation}
And for $s > -1$
\begin{equation}
\label{Good TLCE}
\norm{\norm{2^{js}\psi(2^j)([u\cdot\nabla,\dot{\Delta}_j]\omega)}{l^q(\mathbb{Z})}}{L^p} \lesssim \norm{\nabla u}{L^{\infty}}\norm{\omega}{F^{s,\psi}_{p,q}}+\norm{\omega}{L^{\infty}}\norm{\nabla u}{F^{s,\psi}_{p,q}}.
\end{equation}
\end{proposition}
\begin{proof}We mirror the proof of Proposition 4.1 from \cite{CMZ10}. Using Einstein summation notation over repeated indices $k \in [1,d]$ and a homogeneous paraproduct decomposition, we write
\begin{align*}
[u\cdot\nabla,\dot{\Delta}_j]\omega &= u^k\dot{\Delta}_j\partial_k\omega - \dot{\Delta}_j u^k  \partial_k\omega\\
&= T_{u^k}\dot{\Delta}_j\partial_k\omega + T_{\dot{\Delta}_j\partial_k\omega}u^k + R(u^k,\dot{\Delta}_j\partial_k\omega)
\\ &\hspace{.5cm} - \dot{\Delta}_jT_{u^k}\partial_k\omega-\dot{\Delta}_jT_{\partial_k\omega}u^k-\dot{\Delta}_jR(u^k,\partial_k\omega)
\\
&= [T_{u^k},\dot{\Delta}_j]\partial_k\omega + T'_{\dot{\Delta}_j\partial_k\omega}u^k - \dot{\Delta}_jT_{\partial_k\omega}u^k - \dot{\Delta}_jR(u^k,\partial_k\omega)\\
&\eqdef I + II + III + IV,
\end{align*}
where $T'_uv=T_uv+R(u,v)$.
\\
\\
\textbf{Estimate of $I$:} We write
\begin{align*}
|I| &= \sum_{j' \geq 1}S_{j'-1}u^k\Delta_{j'}\dot{\Delta}_j\partial_k\omega - \dot{\Delta}_j\sum_{j' \geq 1}S_{j'-1}u^k\Delta_{j'}\partial_k\omega\\
&=
\sum_{|j'-j|\leq 4} [S_{j'-1}u^k,\dot{\Delta}_j]\Delta_{j'}\partial_k\omega
\end{align*}
where we used $\dot{\Delta}_{j'}(S_{j-1}f\dot{\Delta}_jf)=0$ for $|j'-j|>4$. Continuing, we expand out the convolution to see that 
\begin{align*}
|I| &= \left| \sum_{|j'-j|\leq 4} \left\{\int_{\R^d}S_{j'-1}u^k(x)-S_{j'-1}u^k(y))2^{jd}\varphi(2^j(x-y))\Delta_{j'}\partial_k\omega(y)dy\right\}\right| \\
&=
 \left| \sum_{|j'-j|\leq 4} \left\{\int_{\R^d}(S_{j'-1}u^k(x)-S_{j'-1}u^k(y))2^{j(d+1)}(\partial_k\varphi)(2^j(x-y))\Delta_{j'}\omega(y)dy\right\}\right| \\
 &\leq
 \sum_{|j'-j|\leq 4} \left\{\norm{\nabla S_{j'-1}u^k}{L^{\infty}}\int_{\R^d} 2^j|x-y|2^{jd}|\nabla \varphi(2^j(x-y))||\Delta_{j'}\omega(y)| dy \right\}\\
 & \leq 
 \sum_{|j'-j| \leq 4} \norm{S_{j'-1}\nabla u^k}{L^{\infty}}|\mathcal{M}\Delta_{j'} \omega(\cdot)|,
\end{align*}
where in the second equality we integrated by parts and used the fact that $u$ is divergence free, in the first inequality we used the mean value theorem, and in the last inequality we used Lemma \ref{convolution bound by maximal function} since $\int_{\mathbb{R}^d}2^{j(d+1)}|x||\varphi(2^jx)|\,dx$ is finite and independent of $j\in\mathbb{Z}$. Thus, multiplying by $2^{js}\psi(2^j)$, taking the $l^q(\mathbb{Z})$ norm in $j$, and then taking the $L^p(\R^d)$ norm we have 
\begin{equation}
\label{TLCE Ii}
\begin{split}
\norm{\norm{2^{js}\psi(2^j)|I|}{l^q(\mathbb{Z})}}{L^p} & \lesssim \norm{\nabla u}{L^{\infty}} \norm{\norm{\sum_{|j'-j| \leq 4} 2^{(j-j')s}\psi(2^j)\mathcal{M}(2^{j's}|\Delta_{j'}\omega|)}{l^q(\mathbb{Z})}}{L^p}\\
&\lesssim \norm{\nabla u}{L^{\infty}} \norm{\norm{\sum_{|j'-j| \leq 4} 2^{(j-j')s}\psi(2^{j'})\mathcal{M}(2^{j's}|\Delta_{j'}\omega|)}{l^q(\mathbb{Z})}}{L^p}.\\
\end{split}
\end{equation} 
Then by Young's convolution inequality for series we have
\begin{equation}
\begin{split}
\norm{\sum_{|j'-j| \leq 4} 2^{(j-j')s}\psi(2^{j'})\mathcal{M}(2^{j's}|\Delta_{j'}\omega|)}{l^q(\mathbb{Z})} 
&\leq 
\left(\sum_{|j| \leq 4} 2^{js}\right)\norm{\mathcal{M}(2^{j's}\psi(2^{j'})|\Delta_{j'}\omega|)}{l^q(\mathbb{Z})}\\
&\lesssim
\norm{2^{j's}\psi(2^{j'})|\Delta_{j'}\omega|}{l^q(\mathbb{Z})},
\end{split}
\end{equation}
where in the last inequality we used Lemma \ref{vector maximal inequality}. Substituting this into \eqref{TLCE Ii} we find that 
\begin{equation}
\label{TLCE I}
\norm{\norm{2^{js}\psi(2^j)|I|}{l^q(\mathbb{Z})}}{L^p}
\lesssim \norm{\nabla u}{L^{\infty}}\norm{\omega}{F^{s,\psi}_{p,q}}.
\end{equation}\\
\\
\textbf{Estimate of $II$:} Using the definition of $II$ we write 
$$
|II| \leq \sum_{j' \geq j-2} |S_{j'+1}\partial_k(\dot{\Delta}_j\omega)\tilde{\Delta}_{j'}u^k| \leq \sum_{j' \geq j-2} \norm{\nabla \dot{\Delta}_j\omega}{L^{\infty}}|\tilde{\Delta}_{j'}u^k|
$$
Then, utilizing Young's convolution inequality for series, we have 
\begin{equation}
\label{TLCE II}
\begin{split}
\norm{\norm{2^{js}\psi(2^j)|II|}{l^q(\mathbb{Z})}}{L^p} &\lesssim \norm{\nabla \omega}{L^{\infty}}\norm{\norm{\sum_{j' \geq j-2} 2^{js}\psi(2^j)|\tilde{\Delta}_{j'}u^k|}{l^q(\mathbb{Z})}}{L^p}\\
&\lesssim
\norm{\nabla \omega}{L^{\infty}}\norm{\norm{\sum_{j' \geq j-2} 2^{(j-j')s}2^{j's}\psi(2^{j'})|\tilde{\Delta}_{j'}u^k|}{l^q(\mathbb{Z})}}{L^p}\\
&\simeq
\norm{\nabla \omega}{L^{\infty}}\norm{\norm{\sum_{j' \geq -1} 2^{(j-j')s}\chi_{\{j-j' \leq 2\}}2^{j's}\psi(2^{j'})|\tilde{\Delta}_{j'}u^k|}{l^q(\mathbb{Z})}}{L^p}\\
&\lesssim
\norm{\nabla \omega}{L^{\infty}}\norm{\left(\sum_{j \geq -1} \chi_{\{j \leq 2\}} 2^{js}\right) \left(\sum_{j'\geq -1}\left(2^{j's}\psi(2^{j'})\tilde{\Delta}_{j'}u^k|\right)^q\right)^{\frac{1}{q}}}{L^p}\\
&\lesssim 
\norm{\nabla \omega}{L^{\infty}}\norm{u}{F^{s,\psi}_{p,q}}.
\end{split}
\end{equation}
\textbf{Estimate of $III$:}
We begin by writing 
$$
|III| 
\leq \sum_{|j'-j| \leq 4} | \dot{\Delta}_j(S_{j'-1}\partial_k\omega \Delta_{j'}u^k)|
\lesssim \sum_{|j'-j| \leq 4} | \mathcal{M}(S_{j'-1}\partial_k\omega\Delta_{j'}u^k)|
\lesssim \norm{\nabla \omega}{L^{\infty}}\sum_{|j'-j| \leq 4 } |\mathcal{M}(|\Delta_{j'}u^k|)|.
$$
Then applying an identical argument that we used to obtain \eqref{TLCE I} we find 
\begin{equation}
\label{TLCE III}
\begin{split}
\norm{\norm{2^{js}\psi(2^j)|III|}{l^q(\mathbb{Z})}}{L^p} & \lesssim \norm{\nabla \omega}{L^{\infty}} \norm{\norm{\sum_{|j'-j| \leq 4} 2^{js}\psi(2^j)\mathcal{M}(|\Delta_{j'} u^k|)}{l^q(\mathbb{Z})}}{L^p}\\
&
\lesssim \norm{\nabla \omega}{L^{\infty}} \norm{\norm{2^{j's}\psi(2^{j'})\mathcal{M}(|\Delta_{j'} u^k|)}{l^q(\mathbb{Z})}}{L^p}\\
&\lesssim \norm{\nabla \omega}{L^{\infty}}\norm{u}{F^{s,\psi}_{p,q}}.
\end{split}
\end{equation}
\textbf{Estimate of $IV$:} Expanding out the convolution we have 
\begin{align*}
|IV| &= \left| \dot{\Delta}_j \sum_{j' \geq -1} \Delta_{j'} \partial_k\omega\tilde{\Delta}_{j'}u^k \right| 
=
\left| \sum_{j' \geq j-3} \dot{\Delta}_j (\Delta_{j'} \partial_k\omega\tilde{\Delta}_{j'}u^k) \right| \\
&=
\left| \sum_{j' \geq j-3} \int_{\R^d} 2^{jd}\varphi(2^j(x-y))(\Delta_{j'}\partial_k\omega\tilde{\Delta}_{j'}u^k)(y)dy\right|.
\end{align*}
Since $u$ is divergence free we then integrate by parts to see that 
\begin{align*}
|IV| &= \left| \sum_{j' \geq j-3} \int_{\R^d} 2^j2^{jd}(\partial_k\varphi)(2^j(x-y)))(\Delta_{j'}\omega\tilde{\Delta}_{j'}u^k)(y)dy\right|
\\
&\lesssim \sum_{j' \geq j-3} 2^j \mathcal{M}(|\Delta_{j'}\omega \tilde{\Delta}_{j'}u^k|)\\
&
\lesssim 
\sum_{j' \geq j-3} 2^j \mathcal{M}(|\Delta_{j'}\omega|)\norm{\tilde\Delta_{j'} u^k}{L^{\infty}}.
\end{align*}
Then by Young's convolution inequality and Bernstein's Lemma we have 
\begin{equation}
\label{TLCE IV}
\begin{split}
\norm{\norm{2^{js}\psi(2^j)|IV|}{l^q}}{L^p} & \lesssim \norm{\norm{\sum_{j' \geq j-3} 2^{j(s+1)} \psi(2^j)\norm{\tilde\Delta_{j'} u^k}{L^{\infty}}\mathcal{M}(|\tilde{\Delta}_{j'}\omega|)}{L^q(\mathbb{Z})}}{L^p}\\
&\lesssim
\norm{\norm{\sum_{j' \geq j-3} 2^{(j-j')(s+1)} \norm{2^{j'}\tilde{\Delta}_{j'}u^k}{L^{\infty}}\mathcal{M}(|2^{j's}\psi(2^{j'})\tilde{\Delta}_{j'}\omega|)}{L^q(\mathbb{Z})}}{L^p}\\
&\lesssim \norm{\nabla u}{L^{\infty}}\norm{\omega}{F^{s,\psi}_{p,q}}.
\end{split}
\end{equation}
Then summing \eqref{TLCE I}, \eqref{TLCE II}, \eqref{TLCE III}, and \eqref{TLCE IV} yields \eqref{Bad TLCE}.
\\
\\
To see why \eqref{Good TLCE} holds we only need to modify the arguments for $II$ and $III$.
\\
\\
\textbf{Second estimate for $II$}
Applying Bernstein's Lemma gives
$$
|II|  \leq \sum_{j' \geq j-2} |S_{j'+1}\partial_k \dot{\Delta}_j\omega \tilde{\Delta}_{j'}u^k| \\
\lesssim 
\sum_{j' \geq j-2} 2^j \norm{\dot{\Delta}_j\omega}{L^{\infty}}|\tilde{\Delta}_{j'}u^k|.
$$
Then since $s+1>0$, by Young's convolution inequality we have
\begin{equation}
\label{TLCE IIii}
\begin{split}
\norm{\norm{2^{js}\psi(2^j)|II|}{l^q(\mathbb{Z})}}{L^p} & \lesssim \norm{\omega}{L^{\infty}}\norm{\norm{\sum_{j' \geq j-2} 2^{j(s+1)}\psi(2^j)|\tilde{\Delta}_{j'}u^k|}{l^q(\mathbb{Z})}}{L^p}\\
&
\lesssim
\norm{\omega}{L^{\infty}} \norm{\norm{\sum_{j' \geq j-2} 2^{(j-j')(s+1)}|2^{j'(s+1)}\psi(2^{j'})\tilde{\Delta}_{j'}u^k| }{l^q(\mathbb{Z})}}{L^p}\\
&\leq \norm{\omega}{L^{\infty}}\left(\sum_{j \leq 2} 2^{j(s+1)}\right)\norm{\norm{|2^{j'(s+1)}\psi(2^{j'})\tilde{\Delta}_{j'}u^k| }{l^q(\mathbb{Z})}}{L^p}\\
&\lesssim \norm{\omega}{L^{\infty}}\norm{u}{F^{s+1,\psi}_{p,q}}.
\end{split}
\end{equation}
\textbf{Second estimate for $III$:} We begin by observing 
\begin{align*}
|III| & \leq \sum_{|j'-j| \leq 4} \mathcal{M}(|S_{j'-1}\partial_k\omega \Delta_{j'}u^k|)\\
&
\lesssim \sum_{|j'-j| \leq 4} 2^{j'} \norm{S_{j'-1}\omega}{L^{\infty}}\mathcal{M}(|\Delta_{j'}u^k|) \\
&
\lesssim 
\norm{\omega}{L^{\infty}}\sum_{|j'-j| \leq 4} 2^{j'} |\mathcal{M}(|\Delta_{j'}u^k|).
\end{align*}
Substituting this in, we have 
\begin{equation}
\label{TLCE IIIii}
\begin{split}
\norm{\norm{2^{js}\psi(2^j)|III|}{l^q(\mathbb{Z})}}{L^p} &\lesssim \norm{\omega}{L^{\infty}}\norm{\norm{\sum_{\substack{j' \geq -1 \\ |j'-j| \leq 4}} 2^{js}\psi(2^j)2^{j'}\mathcal{M}(|\Delta_{j'}u^k|)}{l^q(\mathbb{Z})}}{L^p}\\
&\lesssim 
\norm{\omega}{L^{\infty}}\norm{\norm{\sum_{\substack{j'\geq -1 \\|j'-j| \leq 4}} 2^{(j-j')s}2^{j'(s+1)}\psi(2^{j'}) \mathcal{M}(|\Delta_{j'}u^k|)}{l^q(\mathbb{Z})}}{L^p}\\
&=
\norm{\omega}{L^{\infty}}\norm{\norm{\sum_{\substack{j'\geq -1}}\chi_{|j-j'|\leq 4} 2^{(j-j')s}2^{j'(s+1)}\psi(2^{j'}) \mathcal{M}(|\Delta_{j'}u^k|)}{l^q(\mathbb{Z})}}{L^p}\\
&\lesssim \norm{\omega}{L^{\infty}}\norm{u}{F^{s+1,\psi}_{p,q}}.
\end{split}
\end{equation}
In the second to last inequality we used Young's convolution inequality, and in the last we used Lemma \ref{vector maximal inequality}. Combining \eqref{TLCE I}, \eqref{TLCE IIii}, \eqref{TLCE IIIii}, and \eqref{TLCE IV} establishes the second part of the proposition.
\end{proof}

\begin{lemma}
\label{commutator hammer}{\cite{BCD11}}
Let $\theta \in C^1(\R^d)$ such that $(1+\abs{\cdot})\hat{\theta} \in L^1(\R^d)$. There exists a constant $C>0$ such that for any Lipschitz function $a$ with gradient in $L^p(\R^d)$ and any function $b \in L^q(\R^d)$, we have, for any positive $\lambda$
\begin{equation}
\norm{[\theta(\lambda^{-1}D),a]b}{L^r}\leq C\lambda^{-1}\norm{\nabla a}{L^p}\norm{b}{L^q} \text{   with }\frac{1}{p}+\frac{1}{q}=\frac{1}{r}.
\end{equation}
\end{lemma}
We make use of the following lemma which establishes a bound on the remainder term in Besov spaces of generalized smoothness. The proof follows extremely closely to that of Theorem 2.52 in \cite{BCD11} but we include it for sake of completeness.
\begin{lemma}
\label{remainder estimate}
Let $s_1,s_2 \geq 0$ and $p_1,p_2,r_1,r_2 \in [1,\infty]$. We set 
\begin{equation*}
\frac{1}{p}=\frac{1}{p_1}+\frac{1}{p_2} \leq 1, \frac{1}{r}=\frac{1}{r_1}+\frac{1}{r_2} \leq 1,  \text{ and } s=s_1+s_2>0,
\end{equation*}
then for any $(u,v) \in B^{s_1,\psi}_{p_1,r_1}(\R^d) \times B^{s_2}_{p_2,r_2}(\R^d)$ we have 
\begin{equation}
\norm{R(u,v)}{B^{s,\psi}_{p,r}} \lesssim\norm{u}{B^{s_1,\psi}_{p_1,r_1}}\norm{v}{B^{s_2}_{p_2,r_2}}.
\end{equation}
\end{lemma}
\begin{proof}
We begin by writing
$$
R(u,v) = \sum_{|j'-j|\leq 1} \Delta_{j'}u\Delta_jv = \sum_{j}R_j,
$$
where 
$$
R_j=\sum_{|\nu|\leq 1} \Delta_{j-\nu} u\Delta_j v.
$$
Note that there exists some $R>0$ such that $\supp \widehat{R_j} \subset  2^jB(0,R)$ and $N_0 \in \mathbb{Z}$ such that
$$
j'\geq j+N_0 \implies \Delta_{j'}R_j=0.
$$
We then write
$$
\Delta_{j'}R(u,v) = \sum_{\substack{ |\nu|\leq 1 \\ j\geq j'-N_0}} \Delta_{j'}(\Delta_{j-\nu}u\Delta_jv).
$$
It follows that
\begin{equation}
\begin{split}
2^{j's}\psi(2^{j'})\norm{\Delta_{j'}R(u,v)}{L^p} &\lesssim 2^{j's}\psi(2^{j'})\sum_{\substack{ |\nu|\leq 1 \\ j\geq j'-N_0}} \norm{\Delta_{j-\nu}u\Delta_jv}{L^p}\\
&{\hspace{-3cm}}\lesssim 
\sum_{\substack{ |\nu|\leq 1 \\ j\geq j'-N_0}} 2^{-(j-j')s}2^{(j-\nu)s_1}\psi(2^{j-\nu})\norm{\Delta_{j-\nu}u}{L^{p_1}}2^{js_2}\norm{\Delta_jv}{L^{p_2}}.
\end{split}
\end{equation}
Taking the $\ell^r(\mathbb{Z})$ norm and applying Young's convolution inequality for series and Hölder's inequality yields the desired result.
\end{proof}
\begin{proposition}
\label{besov commutator estimate}
Let $s \geq -1$ and $1 \leq p,r \leq \infty$. Let $u$ be a divergence free vector field over $\R^d$. Define $R_j =[u\cdot \nabla, \Delta_j]\omega$. There exists a constant $C$, depending continuously on $p,s,r,$ and $d$, such that 
$$
\norm{\left(2^{js}\psi(2^j)\norm{\Delta_jR_j}{L^p}\right)_j}{l^r} \leq C\left(\norm{\nabla u}{L^{\infty}}\norm{\omega}{B^{s,\psi}_{p,r}} + \norm{\nabla \omega}{L^{\infty}}\norm{\nabla u}{B^{s-1,\psi}_{p,r}}\right).
$$
Note that this can be restrictive as it may be $\nabla \omega \notin L^{\infty}$. If this is the case we also have that 
$$
\norm{\left(2^{js}\psi(2^j)\norm{\Delta_jR_j}{L^p}\right)_j}{l^r} \leq C\left(\norm{\nabla u}{L^{\infty}}\norm{\omega}{B^{s,\psi}_{p,r}} + \norm{\omega}{L^{\infty}}\norm{\nabla u}{B^{s,\psi}_{p,r}}\right).
$$
\end{proposition}
\begin{proof}
We follow the proof of Lemma 2.100 from \cite{BCD11}. We begin by splitting $u$ into low and high frequencies. We write $u=S_0u + \tilde{u}$. Using Einstein summation notation over repeated indices we write 
$$
R_j = [\tilde u^k,\Delta_j]\partial_k \omega + [S_0u^k, \Delta_j]\partial_k\omega.
$$
The term involving the low frequencies of $u$ is easier to estimate while the other term presents a few challenges. Using paraproducts, we further decompose  $[\tilde u^k,\Delta_j]\partial_k \omega$ as 
\begin{align*}
[\tilde u^k,\Delta_j]\partial_k \omega &=
\tilde{u}^k \cdot \Delta_j \partial_k\omega - \Delta_j (\tilde{u}^k \partial_k \omega)\\
&= T_{\tilde{u}^k}(\Delta_j \partial_k\omega) + T_{\Delta_j \partial_k \omega}(\tilde{u}^k)+R(\tilde{u}^k,\Delta_j\partial_k\omega) -\Delta_j T_{\tilde{u}^k}(\partial_k\omega) -\Delta_j T_{\partial_k\omega}(\tilde{u}^k) -\Delta_jR(\tilde{u}^k,\partial_k\omega)\\
&=
[T_{\tilde{u}^k},\Delta_j]\partial_k \omega T_{\Delta_j \partial_k \omega}(\tilde{u}^k)-\Delta_j T_{\partial_k\omega}(\tilde{u}^k) + +R(\tilde{u}^k,\Delta_j\partial_k\omega) - \Delta_jR(\tilde{u}^k,\partial_k\omega).
\end{align*}
Furthermore, note that 
\begin{align*}
R(\tilde{u}^k,\Delta_j\partial_k\omega) &= \sum_{|l-j'|\leq 1} \sum_{k=1}^d \partial_k \left(\Delta_l\tilde{u}^k\Delta_{j'}\Delta_j\omega\right) - \sum_{|l-j'|\leq 1} \sum_{k=1}^d \Delta_l\partial_k\tilde{u}^k\Delta_{j'}\Delta_j\omega\\
&= \partial_k R(\tilde{u}^k,\Delta_j\omega)-R(\text{div }\tilde{u}, \Delta_j \omega),
\end{align*}
and
\begin{align*}
\Delta_jR(\tilde{u}^k,\partial_k\omega) &= \Delta_j \sum_{|l-j'|\leq 1} \Delta_l \tilde{u}^k \Delta_{j'}\partial_k\omega\\
&=
\Delta_j \partial_k \sum_{|l-k|\leq 1} \Delta_l\tilde{u}^k\Delta_{j'}\omega - \Delta_j \sum_{|l-j'|\leq 1} \Delta_l \partial_k \tilde{u}^k \Delta_{j'}\omega\\
&=
\Delta_j \partial_k R(\tilde{u}^k,\omega) - \Delta_j R(\text{div }\tilde{u},\omega).
\end{align*}
Thus, we decompose
$$
R_j=\sum_{i=1}^8 R_j^i
$$
where 
\begin{align*}
&R_j^1 = [T_{\tilde{u}^k},\Delta_j]\partial_k \omega, \hspace{1cm}&\hspace{1cm} &R_j^2 = T_{\partial_k\Delta_j\omega}\tilde{u}^k, \\
&R_j^3 = -\Delta_jT_{\partial_k\omega}\tilde{u}^k,  \hspace{1cm}&\hspace{1cm} &R_j^4 = \partial_kR(\tilde{u}^k, \Delta_j\omega), \\
&R_j^5 =-R(\text{div }\tilde{u}, \Delta_j\omega),  \hspace{1cm}&\hspace{1cm} &R_j^6 = -\partial_k\Delta_jR(\tilde{u}^k,\omega),\\
&R_j^7=\Delta_jR(\text{div }\tilde{u},\omega),  \hspace{1cm}&\hspace{1cm} &R_j^8 = [S_0u^k,\Delta_j]\partial_k\omega.\\
\end{align*}
In what follows we let $c_j$ denote a sequence such that $\norm{(c_j)}{l^r} \leq 1$. 

\noindent
\textbf{Estimate of $R_j^1$:} Due to the support of the Littlewood-Paley operators we write 
$$
R_j^1 = \sum_{|j-j'| \leq 4} [S_{j'-1}\tilde{u}^k,\Delta_j]\partial_k\Delta_{j'}\omega.
$$
Thus, applying Bernstein's Lemma and Lemma \ref{commutator hammer} we have that 
\begin{equation}
\label{Rj1 estimate}
\begin{split}
2^{js}\psi(2^j)\norm{R_j^1}{L^p} &\lesssim  \norm{\nabla u}{L^{\infty}} \sum_{|j-j'|\leq 4} 2^{js}\psi(2^j)\norm{\Delta_{j'}\partial_k\omega}{L^p}\\
&\lesssim \norm{\nabla u}{L^{\infty}} \sum_{|j-j'|\leq 4} 2^{j'(s-1)}\psi(2^{j'})\norm{\Delta_{j'}\partial_k\omega}{L^p}\\
&\lesssim c_j \norm{\nabla u}{L^{\infty}}\norm{\omega}{B_{p,r}^{s,\psi}}.
\end{split}
\end{equation}
\\
\\
\textbf{Estimate for $R_j^2$:} By the support of the Littlewood-Paley operators we write 
$$
R_j^2 = \sum_{j' \geq j-3} S_{j'-1}\partial_k \Delta_j\omega \Delta_{j'}\tilde{u}^k.
$$
Using Hölder's inequality and Bernstein's Lemma yields
\begin{equation}
\label{Rj2 estimate}
\begin{split}
2^{js}\psi(2^j)\norm{R_j^2}{L^p} &\leq \sum_{j' \geq j-3} 2^{js}\psi(2^j)\norm{S_{j'-1}\partial_k\Delta_j\omega\Delta_{j'}\tilde{u}^k}{L^p}\\
&\lesssim \sum_{j' \geq j-3} 2^{j'(s-1)}\psi(2^j)\norm{S_{j'-1}\partial_k\Delta_j\omega}{L^p}2^{j'}\norm{\Delta_j\tilde{u}^k}{L^{\infty}}\\
&\lesssim \norm{\nabla u}{L^{\infty}}\sum_{j' \geq j-3} 2^{j's}\psi(2^{j'})\norm{\Delta_j\omega}{L^p}\\
&\lesssim c_j \norm{\nabla u }{L^{\infty}}\norm{\omega}{B^{s,\psi}_{p,r}}.
\end{split}
\end{equation}
\\
\\
\textbf{Estimate for $R_j^3$} We write
\begin{align*}
R_j^3 &= -\Delta_jT_{\partial_k\omega}\tilde{u}^k \\
&=
-\sum_{|j-j'|\leq 4}\Delta_j(S_{j'-1}\partial_k\omega\Delta_{j'}\tilde{u}^k).
\end{align*}
Thus we have that 
\begin{equation}
\label{Rj3 estimate}
\begin{split}
2^{js}\psi(2^j)\norm{R_j^3}{L^p} &\lesssim \sum_{|j-j'|\leq 4} 2^{js}\psi(2^j)\norm{\nabla S_{j'-1}\omega}{L^{\infty}}\norm{\Delta_j\tilde{u}^k}{L^p}\\
&\lesssim \norm{\nabla \omega}{L^{\infty}}\sum_{|j-j'|\leq 4} 2^{j'(s-1)}\psi(2^{j'})\norm{\Delta_{j'}\nabla u}{L^p}\\
&\lesssim c_j\norm{\nabla \omega}{L^{\infty}}\norm{\nabla u}{B^{s-1,\psi}_{p,r}}.
\end{split}
\end{equation}
Note we can replace this estimate by 
\begin{equation}
\label{Rj3 estimate 2}
\begin{split}
2^{js}\psi(2^j)\norm{R_j^3}{L^p} &\lesssim \sum_{|j-j'|\leq 4} 2^{js}\psi(2^j)\norm{\nabla S_{j'-1}\omega}{L^{\infty}}\norm{\Delta_j\tilde{u}^k}{L^p}\\
&\lesssim 
\sum_{|j-j'|\leq 4} 2^{j(s+1)}\psi(2^j)\norm{S_{j'-1}\omega}{L^{\infty}}\norm{\Delta_j\tilde{u}^k}{L^p}\\
&\lesssim \norm{\omega}{L^{\infty}}\sum_{|j-j'|\leq 4} 2^{j's}\psi(2^{j'})\norm{\Delta_{j'}\nabla u}{L^{p}}\\
&\lesssim c_j\norm{\omega}{L^{\infty}}\norm{\nabla u}{B^{s,\psi}_{p,r}}.
\end{split}
\end{equation}
\\
\\
\textbf{Bound for $R_j^4$:} We write
$$
R_j^4 = \sum_{|j-j'|\leq 2} \partial_k(\Delta_{j'}\tilde{u}^k\Delta_j\tilde{\Delta}_{j'}\omega).
$$
Using the product rule and Bernstein's Lemma gives 
\begin{equation}
\label{Rj4 estimate}
\begin{split}
2^{js}\psi(2^j) \norm{R_j^4}{L^p} &\lesssim \sum_{|j-j'|\leq 2} 2^{js}\psi(2^j)\norm{\partial_k\left( \Delta_{j'}\tilde{u}^k\Delta_j\tilde{\Delta}_{j}\omega\right)}{L^p}\\
&\lesssim \sum_{|j-j'|\leq 2} 2^{j's}\psi(2^{j'})\left( \norm{\nabla u}{L^{\infty}}\norm{\Delta_{j}\omega}{L^p}+\norm{\Delta_{j'}\tilde{u}^k}{L^{\infty}}\norm{\Delta_j\tilde{\Delta}_j \partial_k\omega}{L^p}\right)\\
&\lesssim 
\sum_{|j-j'|\leq 2} 2^{j's}\psi(2^j) \left(\norm{\nabla u}{L^{\infty}}\norm{\Delta_{j}\omega}{L^p} + 2^{j'}\norm{\Delta_{j'}\tilde{u}^k}{L^p}\norm{\Delta_j \omega}{L^p}\right)
\\
&\lesssim c_j\norm{\nabla u}{L^{\infty}}\norm{\omega}{B^{s,\psi}_{p,r}}.
\end{split}
\end{equation}
\\
\\
\textbf{Estimate for $R_j^5$:} We write
$$
R_j^5 = \sum_{|j-j'|\leq 2} \Delta_{j'}\text{div }\tilde{u}\Delta_j\tilde{\Delta}_{j}\omega.
$$
Thus we have that 
\begin{equation}
\label{Rj5 estimate}
\begin{split}
2^{js}\psi(2^j)\norm{R_j^5}{L^p} &\lesssim \sum_{|j-j'|\leq 2} 2^{j's}\psi(2^{j'})\norm{\Delta_{j'}\nabla u}{L^{\infty}}\norm{\tilde{\Delta}_j \omega}{L^p}\\
&\lesssim 
c_j\norm{\nabla u}{L^{\infty}}\norm{\omega}{B^{s,\psi}_{p,r}}.
\end{split}
\end{equation}
\\
\\
\textbf{Estimate for $R_j^6$: }Recall that $R_j^6 = -\partial_k\Delta_jR(\tilde{u},\omega)$. By Lemma \ref{remainder estimate} and Bernstein's lemma
\begin{equation}
\label{Rj6 estimate}
\norm{\left(2^{js}\psi(2^j)\norm{\Delta_jR_j^6}{L^p}\right)_j}{l^r} \lesssim \norm{\omega}{B^{s,\psi}_{p,r}}\norm{\nabla u}{B_{\infty,\infty}^0} \lesssim \norm{\omega}{B^{s,\psi}_{p,r}}\norm{\nabla u}{L^{\infty}},
\end{equation}
where in the last inequality we used the embedding $L^{\infty} \hookrightarrow B_{\infty,\infty}^0$.
\\
\\
\textbf{Estimate for $R_j^7$: } Arguing as just before we have that
\begin{equation}
\label{Rj7 estimate}
\norm{\left(2^{js}\psi(2^j)\norm{\Delta_jR_j^7}{L^p}\right)_j}{l^r} \lesssim \norm{\omega}{B^{s,\psi}_{p,r}}\norm{\nabla u}{L^{\infty}}.
\end{equation}
\\
\\
\textbf{Estimate for $R_j^8$: }Recall that $R_j^8 = [\Delta_{-1}\tilde{u}^k,\Delta_j]\partial_k\omega$. Since
$$
R_j^8 = \sum_{|j-j'|\leq 1}[\Delta_j, \Delta_{-1}u]\cdot\nabla\Delta_{j'}\omega,
$$
applying Lemma \ref{commutator hammer} yields
\begin{equation}
\label{Rj8 estimate}
\begin{split}
2^{j}\psi(2^j)\norm{R_j^8}{L^p} &\lesssim \sum_{|j-j'|\leq 1} 2^{-j'}2^{j's}\psi(2^{j'})\norm{\nabla \Delta_{-1}u}{L^{\infty}}\norm{\nabla \Delta_{j'}\omega}{L^p}\\
&\lesssim c_j\norm{\nabla u}{L^{\infty}}\norm{\omega}{B_{p,r}^{s,\psi}}.
\end{split}
\end{equation}
Thus, combining \eqref{Rj1 estimate}-\eqref{Rj8 estimate} and using that $\norm{c_j}{l^r} \leq 1$ yields the desired result.
\end{proof}

\subsection{Multiplier Theorems}
When estimating the Biot-Savart law in Triebel-Lizorkin spaces of generalized smoothness the following proposition will be useful. 

\begin{proposition}
\label{TL multiplier theorem} Let $s,a \in \R$,  $0<p,q<\infty$, and let $N=\left[\frac{d}{2} + \frac{d}{\min \{p,q\}}\right] +1$. Assume that $\sigma$ is a $C^N$ function on $\R^d\setminus \{0\}$ that satisfies
\begin{equation}
\label{multiplier assumption}
|\partial^{\gamma}\sigma(\xi)| \leq C_{\gamma}|\xi|^{-|\gamma|-a}
\end{equation}
for all $|\gamma|\leq N$. Then there exists a constant $C>0$ such that for all $f \in \mathcal{S}'(\R^d)$ we have 
$$
\norm{\mathcal{F}^{-1}(\sigma \hat{f})}{\dot{F}^{s,\psi}_{p,q}} \leq C \norm{f}{\dot{F}^{s-a,\psi}_{p,q}}.
$$
\end{proposition}
Before proving Proposition \ref{TL multiplier theorem} we need the following lemma.
\begin{lemma}\cite{G14}
\label{technical tl lemma} Let $0<c_0<\infty$ and $0<r<\infty$. Then there exists constants $C_1$ and $C_2$ (which depends on $d,c_0$ and $r$) such that for all $t>0$ and all $u \in C^1(\R^d)$ whose Fourier transform is supported in the ball $|\xi| \leq c_0t$ and that satisfy $|u(z)| \leq B(1+|z|)^{\frac{d}{r}}$ for some $B>0$ we have the estimate
$$
\sup_{z \in \R^d} \frac{1}{t} \frac{|\nabla u(x-z)|}{(1+t|z|)^{\frac{d}{r}}} \leq C_1 \sup_{z \in \R^d} \frac{|u(x-z)|}{(1+t|z|)^{\frac{d}{r}}} \leq C_2 \mathcal{M}(|u|^r)(x)^{\frac{1}{r}},
$$
where $\mathcal{M}$ denotes the Hardy-Littlewood maximal operator.
\end{lemma}
\begin{proof}[Proof of Proposition \ref{TL multiplier theorem}] We begin by writing $\sigma_j(\xi)=\widetilde{\varphi}_j\sigma(\xi)$. Then observe that 
$$
\Delta_j(\sigma(D)f)=\Delta_j(\sigma_j(D)f)=\sigma_j(D)\Delta_jf.
$$
Let $r=\min\{p,q\}$ and $N=\frac{d}{2}+\frac{d}{r}+\delta$ for some $\delta>0$. We then estimate 
\begin{align*}
&|\sigma_j(D)\Delta_jf(x)| = \left| \int_{\R^d} \widecheck{\sigma_j}(y)\Delta_jf(x-y)\frac{(1+2^j|y|)^{\frac{d}{r}}}{(1+2^j|y|)^{\frac{d}{r}}}dy\right|\\
&\leq \sup_{y \in \R^d} \frac{|(\Delta_jf)(x-y)|}{(1+2^j|y|)^{\frac{d}{r}}} \int_{\R^d} |\widecheck{\sigma_j}(y)| (1+2^j|y|)^{\frac{d}{r}} \left(\frac{1+2^j|y|}{1+2^j|y|}\right)^{\frac{d}{2}+\delta}dy\\
&\lesssim \left(\mathcal{M}(|\Delta_jf|^r)(x)\right)^{\frac{1}{r}} \left(\int_{\R^d} |\widecheck{\sigma_j}(y)|^2(1+2^j|y|)^{\frac{2d}{r}+d+2\delta}dy \times\int_{\R^d}(1+2^j|y|)^{-d-2\delta}dy\right)^{\frac{1}{2}}
\end{align*}
where in the final inequality we used Lemma \ref{technical tl lemma} and Cauchy Schwarz inequality. Furthermore
$$
\int_{\R^d}(1+2^j|y|)^{-d-2\delta}dy \lesssim 2^{-jd}.
$$
Thus, 
\begin{align*}
&|\sigma_j(D)\Delta_jf(x)| \lesssim \left(\mathcal{M}(|\Delta_jf|^r)(x)\right)^{\frac{1}{r}} \left(\int_{\R^d} 2^{-jd}|\widecheck{\sigma_j}(y)|^2(1+2^j|y|)^{\frac{2d}{r}+d+2\delta}dy \right)^{\frac{1}{2}}\\
&\leq
 \left(\mathcal{M}(|\Delta_jf|^r)(x)\right)^{\frac{1}{r}} \left(\int_{\R^d} |\widecheck{\sigma_j(2^j(\cdot))}(y)|^2(1+|y|)^{2N}dy \right)^{\frac{1}{2}}\\
 &\lesssim
  \left(\mathcal{M}(|\Delta_jf|^r)(x)\right)^{\frac{1}{r}} \left(\int_{\R^d} |\sum_{|\gamma| \leq N} y^{|\gamma|}\widecheck{\sigma_j(2^j(\cdot))}(y)|^2dy \right)^{\frac{1}{2}}\\
  &=  \left(\mathcal{M}(|\Delta_jf|^r)(x)\right)^{\frac{1}{r}} \left(\int_{\R^d} |\sum_{|\gamma| \leq N} \partial^{\gamma}\sigma_j(2^j\xi)|^2d\xi \right)^{\frac{1}{2}}\\
  &\lesssim 2^{-ja}\left(\mathcal{M}(|\Delta_jf|^r)(x)\right)^{\frac{1}{r}}.
\end{align*}
where in the last inequality we used that $\varphi$ is Schwartz function, and that $\sigma_j$ is compactly supported. Summarizing, we have shown that 
\begin{equation}
\label{tl multiplier proof pre lq}
2^{js}\psi(2^j)|\Delta_j(\sigma(D)f)| \lesssim 2^{j(s-a)}\psi(2^j)\left(\mathcal{M}(|\Delta_jf|^r)(x)\right)^{\frac{1}{r}}
\end{equation}
Taking the $l^q(\mathbb{Z})$ norm in $j$ of \eqref{tl multiplier proof pre lq} and then applying Lemma \ref{vector maximal inequality} yields
\begin{align*}
\norm{\left(\sum_{j \in \mathbb{Z}} \left(2^{js}\psi(2^j)|\Delta_j(\sigma(D)f)|\right)^q\right)^{\frac{1}{q}}}{L^p} &\lesssim
\norm{\left( \sum_{j \in \mathbb{Z}} \left(\left(\mathcal{M}(2^{j(s-a)}\psi(2^j) |\Delta_jf|)^r(x)\right)^{\frac{1}{r}}\right)^q \right)^{\frac{1}{q}}}{L^p}\\
&=
\norm{\norm{\mathcal{M}(|2^{j(s-a)}\psi(2^j) \Delta_jf|^r)}{l^{\frac{q}{r}}(\mathbb{Z})}}{L^{\frac{p}{r}}(\R^d)}^{\frac{1}{r}}\\
&\lesssim 
\norm{\norm{(|2^{j(s-a)}\psi(2^j) \Delta_jf|^r)}{l^{\frac{q}{r}}(\mathbb{Z})}}{L^{\frac{p}{r}}(\R^d)}^{\frac{1}{r}}\\
&\leq
\norm{f}{\dot{F}^{s-a,\psi}_{p,q}},
\end{align*}
where in the last inequality we used Minkowski's integral inequality.
\end{proof}

\bibliographystyle{alpha}
\bibliography{wellposedness}

\begin{thebibliography}{BKM84}

\bibitem[BCD11]{BCD11}
Hajer Bahouri, Jean-Yves Chemin, and Raphael Danchin.
\newblock {\em Fourier Analysis and Nonlinear Partial Differential Equations}, volume 343 of {\em Grundlehren der mathematischen Wissenschaften}.
\newblock Springer, 2011.

\bibitem[BGT87]{BGT87}
N.~H. Bingham, C.~M. Goldie, and J.~L. Teugels.
\newblock {\em Regular Variation}.
\newblock Cambridge University Press, 1987.

\bibitem[BKM84]{BKM84}
J~Thomas Beale, Tosio Kato, and Andrew Majda.
\newblock Remarks on the breakdown of smooth solutions for the {3-D} {Euler} equations.
\newblock {\em Comm. Math. Phys.}, 94(1):61--66, 1984.

\bibitem[BKS13]{BKS13}
V.V. Buldygin, O.I. Klesov, and J.G. Steinebach.
\newblock Equivalent monotone versions of prv functions.
\newblock {\em Journal of Mathematical Analysis and Applications}, 401(2):526--533, 2013.

\bibitem[BL15]{BL15}
Jean Bourgain and Dong Li.
\newblock Strong ill-posedness of the incompressible euler equation in borderline sobolev spaces.
\newblock {\em Inventiones mathematicae}, 201:97--157, 2015.

\bibitem[BN19]{BN19}
Elia Bru{\'e} and Quoc-Hung Nguyen.
\newblock On the sobolev space of functions with derivative of logarithmic order.
\newblock {\em Advances in Nonlinear Analysis}, 9(1):836--849, 2019.

\bibitem[BN21]{BN21}
Elia Bru{\'e} and Quoc-Hung Nguyen.
\newblock Sharp regularity estimates for solutions of the continuity equation drifted by sobolev vector fields.
\newblock {\em Anal. PDE}, 14(1):2539--2559, 21.

\bibitem[CCS24]{CCS24}
Gennaro Ciampa, Gianluca Crippa, and Stefano Spirito.
\newblock Propagation of logarithmic regularity and inviscid limit for the 2d euler equations.
\newblock {\em arXiv preprint arXiv:2402.07622}, 2024.

\bibitem[Cha03]{DC03}
D.~Chae.
\newblock On the euler equations in the critical triebel-lizorkin spaces.
\newblock {\em Archive for Rational Mechanics and Analysis}, 170:185--210, 2003.

\bibitem[Cha04]{DC04}
Dongho Chae.
\newblock Local existence and blow-up criterion for the euler equations in the besov spaces.
\newblock {\em Asymptotic Analysis}, 38(3-4):339--358, 2004.

\bibitem[Che95]{CH95}
Jean-Yves Chemin.
\newblock {\em Perfect Incompressible Fluids}.
\newblock Oxford University Press, 1995.

\bibitem[CMZ10]{CMZ10}
Qionglei Chen, Changxing Miao, and Zhifei Zhang.
\newblock On the well-posedness of the ideal mhd equations in the triebel-lizorkin spaces.
\newblock {\em Archive for Rational Mechanics and Analysis}, 195:561--578, 2010.

\bibitem[dM01]{Moura01}
Susana~Domingues de~Moura.
\newblock {\em Function spaces of generalised smoothness}.
\newblock 2001.

\bibitem[DT18]{DT18}
Oscar Dominguez and Sergey Tikhonov.
\newblock {\em Function Spaces of Logarithmic Smoothness: Embeddings and Characterizations}, volume 282.
\newblock 2018.

\bibitem[FS71]{FS71}
Charles Fefferman and Elias Stein.
\newblock Some maximal inequalities.
\newblock {\em American Journal of Mathematics}, 93:107--115, 1971.

\bibitem[Gra14]{G14}
Loukas Grafakos.
\newblock {\em Modern Fourier Analysis}.
\newblock Graduate Texts in Mathematics. Springer, New York, NY, 2014.

\bibitem[Kar30]{K30}
Jovan Karamata.
\newblock Sur certains “tauberian theorems” de m. m. hardy et littlewood.
\newblock {\em Mathematica (Cluj)}, 3:33--48, 1930.

\bibitem[KP86]{KP86}
Tosio Kato and Gustavo Ponce.
\newblock Commutator estimates and the euler and navier stokes equations.
\newblock {\em Communications on Pure and Applied Mathematics}, 41(7):891--907, 1986.

\bibitem[Lic25]{L25}
Leon Lichtenstein.
\newblock {\"U}ber einige existenzprobleme der hydrodynamik.
\newblock {\em Mathematische Zeitschrift und Physik}, 23:89--154, 1925.

\bibitem[MS22]{MS22}
David Meyer and Christian Seis.
\newblock Propagation of regularity for transport equations. a littlewood-paley approach.
\newblock {\em Indiana University Mathematics Journal}, 2022.

\bibitem[Ste70]{S70}
Elias~M. Stein.
\newblock {\em Singular Integrals and Differentiability Properties of Functions}.
\newblock Princeton University Press, Princeton, 1970.

\bibitem[Tri97]{T97}
Hans Triebel.
\newblock {\em Fractals and Spectra Related to Fourier Analysis and Function Spaces}.
\newblock Birkh\"auser, 1997.

\bibitem[Vis98]{MV98}
M.~Vishik.
\newblock Hydrodynamics in besov spaces.
\newblock {\em Archive for Rational Mechanics and Analysis}, 145:197--214, 1998.

\bibitem[Wol33]{WW33}
W.~Wolibner.
\newblock Un théorème sur l’existence du mouvement plan d’un fluide parfait, homogène, incompressible, pendant un temps infiniment long.
\newblock {\em Mathematische Zeitschrift}, 37:698--726, 1933.

\end{thebibliography}

\end{document}